%% file: main.tex
\documentclass[11pt,english]{article}
\usepackage[utf8]{inputenc}
\usepackage[margin= 22mm,bottom=25mm, top= 22mm]{geometry}

\input{preamble}
\input{macros}

\title{Rainbow Hamiltonicity in uniformly coloured perturbed digraphs}

\author{
	Kyriakos Katsamaktsis \thanks{
	Department of Mathematics, 
	University College London, 
	Gower Street, London WC1E~6BT, UK. 
	Email: \texttt{kyriakos.katsamaktsis.21}@\texttt{ucl.ac.uk}.
    Research supported by the Engineering and Physical Sciences Research Council [grant number EP/W523835/1].
	}
	\and
	Shoham Letzter\thanks{
	Department of Mathematics, 
	University College London, 
	Gower Street, London WC1E~6BT, UK. 
	Email: \texttt{s.letzter}@\texttt{ucl.ac.uk}. 
	Research supported by the Royal Society.
	}
    \and
    Amedeo Sgueglia\thanks{
    Department of Mathematics, 
	University College London, 
	Gower Street, London WC1E~6BT, UK. 
	Email: \texttt{a.sgueglia}@\texttt{ucl.ac.uk}. 
	Research supported by the Royal Society.
    }
}

\begin{document}
\date{}
\maketitle

\begin{abstract}
\setlength{\parskip}{\medskipamount}
\setlength{\parindent}{0pt}
\noindent

    We investigate the existence of a rainbow Hamilton cycle in a uniformly edge-coloured randomly perturbed digraph. We show that for every \(\delta \in (0,1)\) there exists $C = C(\delta) > 0$ such that the following holds. Let $D_0$ be an $n$-vertex digraph with minimum semidegree at least $\delta n$ and suppose that each edge of the union of $D_0$ with a copy of the random digraph $\dnp{n}{C/n}$ on the same vertex set gets a colour in \([n]\) independently and uniformly at random. Then, with high probability, $D_0 \cup \dnp{n}{C/n}$ has a rainbow directed Hamilton cycle.
    
    This improves a result of Aigner-Horev and Hefetz (2021), who proved the same in the undirected setting when the edges are coloured uniformly in a set of $(1 + \eps)n$ colours.
\end{abstract}

\section{Introduction}

	Determining which minimum degree forces the containment of a given spanning subgraph is a central theme in extremal combinatorics.
	The prototypical example is  Dirac's theorem~\cite{dirac}, which 
	says that for an $n$-vertex graph $G$ the condition $\delta(G) \ge n/2$ is sufficient to guarantee that $G$ is Hamiltonian, and that this minimum degree condition is best possible.
	The question has been investigated in the setting of digraphs as well. Here a natural notion of minimum degree is
	the minimum \emph{semidegree}, denoted by $\delta^0(D)$, which is the minimum over all in- and out-degrees of the vertices of $D$.
	Ghouila-Houri~\cite{ghouila-houri} proved a directed analogue of Dirac's theorem, showing that, if $D$ is an $n$-vertex digraph with $\delta^0(D) \ge n/2$, then $D$ is Hamiltonian. 
	Here by a digraph $D$ being \emph{Hamiltonian} we mean that $D$ contains a \emph{directed Hamilton cycle}, i.e.\ a Hamilton cycle with all edges oriented consistently.

	On the other hand, one of the main pursuits in probabilistic combinatorics is understanding the minimum \(p\) such that
	\(\gnp{n}{p}\), the binomial random graph on $[n]$ with edge probability \(p\), contains a given subgraph 
	\emph{with high probability}\footnote{
	We say that a sequence of events \((A_n)_{n \in \nats}\) holds \emph{with high probability} if \(\prob{A_n} \tendsto 1\) as \(n\tendsto \infty\).
	}.
	Following the breakthrough of P\'osa~\cite{posa}, 
	it was proven~\cite{komlos83,korshunov1977}
	that, with high probability,
	$\gnp{n}{p}$ is Hamiltonian if it has minimum degree at least $2$, implying that $\log n / n$ is a \emph{sharp threshold} for Hamiltonicity. Here $\hat{p} = \hat{p}(n)$ is said to be a \emph{sharp threshold} for a graph property $\calP$ if, for every $\eps > 0$, with high probability the following holds.
	\begin{equation*}
		\begin{array}{ll}
			G(n,p) \notin \calP & \text{ if } p \le (1-\eps)\hat{p}, \\
			G(n,p) \in \calP & \text{ if } p \ge (1+\eps)\hat{p}. 
		\end{array}
	\end{equation*}
	The directed analogue of \(\gnp{n}{p}\), which we call the \emph{random directed graph} and denote by \(\dnp{n}{p}\), is the random digraph on $[n]$ where each ordered pair of distinct vertices forms a directed edge independently with probability \(p\).
	McDiarmid showed~\cite{mcdiarmid} that the threshold for Hamiltonicity in \(\dnp{n}{p}\) is at most that of \(\gnp{n}{p}\), and Frieze~\cite{frieze-dnp-hc} proved a directed analogue of the above result about Hamiltonicity of random graphs with minimum degree at least $2$, from which it follows that $\log n / n$ is a sharp threshold for Hamiltonicity for directed random graphs as well.
	Recently, Montgomery~\cite{mont-dnp-hc}
	determined sharp thresholds for all possible orientations of a Hamilton cycle in $\dnp{n}{p}$.

	As an interpolation between the deterministic and probabilistic models,
	Bohman, Frieze and Martin~\cite{bohman} introduced the  \emph{perturbed} model for graphs and digraphs.
	Given a fixed $\delta > 0$,
    let \(D_0\) be a digraph on vertex set \([n]\) with minimum semidegree at least \(\delta n\).
    The perturbed digraph is \(D_0 \cup \dnp{n}{p}\), i.e.\ it is the union of a digraph on vertex set $[n]$ with minimum semidegree at least $\delta n$, and the random digraph $\dnp{n}{p}$ on the same vertex set.
	The perturbed graph model is defined similarly as $G_0 \cup \gnp{n}{p}$, where $G_0$ is an $n$-vertex graph with minimum degree at least $\delta n$.
	In~\cite{bohman} they proved that
	there exists \(C>0\), depending only on \(\delta\), such that for each $n$-vertex digraph \(D_0\) with minimum semidegree at least \(\delta n\),
	the perturbed digraph \( D_0 \cup \dnp{n}{C/n}\) has 
	with high probability a directed Hamilton cycle.
	That is, for every digraph with linear minimum semidegree, adding linearly many random edges results in a graph that with high probability contains a directed Hamilton cycle.
	Up to the dependence on $\delta$ and $C$, this is best possible for all $\delta \in (0, 1/2)$, since the complete bipartite digraph with parts of size \(\delta n\) and \((1 - \delta)n\), with each edge appearing with both orientations, requires \(\Omega(n)\) edges to be Hamiltonian. (When $\delta \ge 1/2$ no random edges are needed, due to Ghouila-Houri's theorem.)
	This was generalised very recently by Araujo, Balogh, Krueger, Piga and Treglown~\cite{simon-all-orientations}, who showed that with high probability the same hypotheses ensure that the perturbed digraph contains every orientation of a cycle of every possible length, simultaneously.

	In this paper we consider a rainbow variant of the theorem of Bohman, Frieze and Martin.
    A subset of the edges of an edge-coloured graph or digraph is called \emph{rainbow} if no two edges share a colour, and a subgraph  is called rainbow if  its edge set is rainbow.
	For a finite set of colours \(\cols\),  
	a graph or digraph is \emph{uniformly coloured} in \(\cols\) if each edge gets a colour in \(\cols\) independently and uniformly at random.
	The problem of finding rainbow subgraphs of uniformly coloured graphs is well studied, in particular for \(\gnp{n}{p}\) \cite{cooper-frieze,ferber-kriv,frieze-loh,ferber-optimal-colours,frieze-bal-optimal-colours}.
	The analogous problem in perturbed graphs was considered only more recently~\cite{rainbow-perturbed-trees, elad-hefetz,anastos-first}.
	In particular, the problem of containing a rainbow Hamilton cycle
	was first addressed by
	Anastos and Frieze~\cite{anastos-first}, who showed that if the number of colours is at least about \(120 n\), then \(G\sim G_0 \cup \gnp{n}{C/n}\) has a rainbow Hamilton cycle with high probability, for suitable \(C\) depending only on \(\delta\).
	Later, Aigner-Horev and Hefetz~\cite{elad-hefetz} improved this result by showing that, at the same edge probability in the random graph,
	 \((1+\eps)n\) colours suffice for every $\eps > 0$ (where $C$ now depends also on $\eps$).
	We prove that the optimal number of colours suffices.

	\begin{theorem}\label{undirected-main-thm}
		For any \(\delta \in (0,1/2)\) there exists \(C>0\) such that 
		the following holds.
		Let \(G_0\) be an $n$-vertex graph with minimum degree at least $\delta n$ and let \(G \sim G_0\cup \gnp{n}{C/n}\) be uniformly coloured in \([n]\).
		Then with high probability \(G\) contains a rainbow  Hamilton cycle.
	\end{theorem}

	In fact, we can prove the corresponding result for uniformly coloured perturbed digraphs.

	\begin{theorem}\label{main-thm}
	For any \(\delta \in (0,1/2)\) there exists \(C>0\) such that 
	the following holds.
	Let \(D_0\) be an $n$-vertex digraph with minimum semidegree at least $\delta n$ and let
	 \(D \sim D_0\cup \dnp{n}{C/n}\) be uniformly coloured in \([n]\).
	 Then with high probability \(D\) contains a rainbow directed Hamilton cycle.
	\end{theorem}

	As explained above, for $\delta \in (0,1/2)$, both results have the optimal edge probability, up to the dependence of $C$ on $\delta$.
	We remark that the first two authors proved \Cref{undirected-main-thm} in unpublished work\footnote{See \href{https://arxiv.org/abs/2304.09155}{https://arxiv.org/abs/2304.09155v1}.}, using a somewhat simpler version of \Cref{lemma:main}.
    Here, we will deduce 
	\Cref{undirected-main-thm} from \Cref{main-thm} by a variation of McDiarmid's celebrated coupling argument~\cite{mcdiarmid} (see \Cref{section:proof-main-result}).
	
	The paper is structured as follows.
	In \Cref{section:proof-sketch} we sketch the proof of \Cref{main-thm}.
	In \Cref{section:proof-main-result} we prove \Cref{main-thm} assuming \Cref{lemma:main}, the key lemma of the paper, and deduce \Cref{undirected-main-thm} from \Cref{main-thm}.
	In \Cref{section:prelims} we state and prove some preliminary results that we will need later.
	\Cref{section:absorber} is the most technical, where we prove the existence of the `gadgets' which underpin \Cref{lemma:main}.
	In \Cref{section:proof-main-lemma} we prove \Cref{lemma:main}.
    We finish with some concluding remarks in \Cref{sec:conclusion}.

	\textbf{Notation.}
		Given a digraph \(D\) and $x,y \in V(D)$, we write \(xy \) for the edge directed from \(x\) to \(y\).
		Given
		\(X, Y\subseteq V(D)\), we write \(E_D(X,Y)\) for the set of  edges \(xy\) with \(x\in X\) and \(y\in Y\), and \(e_D^+(X,Y) = |E_D(X,Y)|\) for its size, and similarly \(e_D^-(X,Y) = |E_D(Y,X)|\).
		The out-neighbourhood of a vertex \(v\) is denoted by \(N_D^+(v)\) and its size by $\deg_D^+(v)$.
		For a vertex set \(Y\), we write \(N_D^+(v, Y) = N_D^+(v) \cap Y\) and denote its size by \(\deg_D^+(v,Y) = |N_D^+(v, Y)|\).
		Similarly we will use $N_D^-(v), \deg_D^-(v)$, $N_D^-(v,Y)$ and \(\deg^-_D(v,Y)\).
		We will suppress \(D\) when the digraph in question is clear.
		We recall that the \emph{length} of a (directed) path is the number of its edges.
		Given an edge-coloured digraph $D$, we denote the colour of an edge $e$ by $\cols(e)$ and the set of colours on the edges of a subdigraph \(D'\) by \(\cols(D')\).
		Moreover we say that \(D'\) is \emph{spanning} in a colour set \(\cols'\) if \(\cols(D') = \cols'\).
		
        Throughout the paper,
		we will  assume that \(n\) is sufficiently large.
		Asymptotic notation hides absolute constants: if for some \(x, \eps, n>0\) we write \(x=O(\eps n)\), then there is an absolute constant \(C>0\), which does not depend on \( x, \eps, n\) or any other parameters, such that \(x\le C \eps n\).
        We write \(x= \Omega(y)\) if \(y=O(x)\), and we write \(x=\Theta(y)\) if both \(x=\Omega(y)\) and \(x=O(y)\).
		For $x,y \in (0,1)$, we write \(x \ll y\) if \(x<f(y)\) for an implicit positive increasing function $f$.

\section{Proof sketch} \label{section:proof-sketch}

	Our proof uses the \emph{absorption method}.
	This method is typically applicable when one searches for a spanning subgraph, and involves two stages: finding an almost spanning subgraph; and dealing with the remainder, by having a `special' set of vertices, put aside at the beginning, that can cover any sufficiently small set of vertices.
	For finding the special set of vertices, we prove \Cref{lemma:main}.
    It states that with high probability
	the perturbed digraph has a subgraph \(\Habs\) such that, for any small sets of vertices \(V'\) and colours \(\cols'\) with \(\abs{V'} = \abs{\cols'}\), disjoint from the vertices and colours of \(\Habs\) the following holds:
	there exists a rainbow directed path \(Q\) with
	vertex set \(V' \cup V(\Habs)\) and colours \(\cols' \cup \cols(\Habs)\), whose ends can be chosen arbitrarily within \(V'\).

	Let us briefly sketch the proof of \Cref{lemma:main}.
	We first put aside a subset of the vertices and a subset of the colours, which are typically called the `reservoir' or the `flexible set',
	that have the following property:  for any sets of vertices and colours \(V'\) and \(\cols'\) of the same small size (much smaller than the reservoir) which are disjoint from the reservoir,  we can find a rainbow directed path \(Q_0\) that
	uses \(V'\), \(\cols'\), and a set of $O(|V'|)$ vertices and colours from the reservoir.
	The question is then how to cover the rest of the reservoir; to this end, we build an `absorbing structure' (\(\Habs\) above) which has the following property: it can `absorb' any subset of vertices and colours of the same size of the reservoir in a rainbow directed path \(\Qabs\). Then combining \(\Qabs\) and 
	\(Q_0\) gives \(Q\). 

	The `absorbing structure' \(\Habs\) is built by putting together several `gadgets' or `absorbers', graphs on \(\Theta(1)\) vertices with the following property:
	each gadget has two directed paths with the same endpoints such that one avoids a designated vertex $v$ and colour $c$ in the reservoir, and the other one `absorbs' $v$ and $c$ (see \Cref{fig:absorber} in \Cref{section:absorber}).
	This absorbing structure is based on one that was introduced by  Gould, Kelly, K\"uhn and Osthus~\cite{kuhn-osthus} for constructing rainbow Hamilton paths in random optimal colourings of the (undirected) complete graph, which in turn is based on ideas of Montgomery \cite{montgomery}.

\section{Proof of \Cref{main-thm} and \Cref{undirected-main-thm}} \label{section:proof-main-result}

	In this section we prove the main theorem, \Cref{main-thm}, and use it to deduce \Cref{undirected-main-thm}.
	We will use \Cref{lemma:main} below, which we prove in \Cref{section:proof-main-lemma}.

	\begin{lemma} \label{lemma:main}
		Let \(
		1/n \ll 1/C \ll \eta \ll \gamma \ll \delta
		\).
		Let $D_0$ be a digraph on \(n\) vertices with minimum semidegree at least \(\delta n\),
		and suppose that \(D \sim D_0\cup \dnp{n}{C/n}\) is uniformly coloured in \(\cols = [n]\).
		Then, with high probability, 
		\(D\) contains a digraph \(\Habs \) on at most
		\(\gamma n\) vertices with the following property.
		For any 
		\(V' \subseteq V\setminus V(\Habs)\), 
		\(\cols ' \subseteq \cols \setminus \cols(\Habs) \) with 
		\(2 \le \abs{V'} = \abs{\cols'}\le \eta n\)
		and distinct \(x,y\in V'\),
		there exists a rainbow directed path \(Q\) such that
		\begin{itemize}
			\item 
				$Q$ is a path from $x$ to $y$,
			\item \(V(Q) = V(\Habs) \cup V'\),
			\item \(\cols(Q) = \cols(\Habs) \cup \cols'\).
		\end{itemize}
	\end{lemma}

	The next lemma is a rainbow directed version of a commonly used consequence of the depth first search algorithm~\cite{dfs-pseudorandom}, which we will use to find an almost spanning rainbow directed path.
	The undirected version of the lemma was used in \cite[Lemma 2.17]{ferber-kriv} in the binomial random graph and in~\cite[Proposition 2.1]{elad-hefetz} for finding a Hamilton cycle in the undirected perturbed graph.
	\Cref{lemma:rdfs-lemma} can be proven by following almost verbatim the proof of ~\cite[Proposition 2.1]{elad-hefetz}.

	\begin{lemma}[cf.\ Proposition 2.1~\cite{elad-hefetz} and  Lemma 2.17~\cite{ferber-kriv}] \label{lemma:rdfs-lemma}
		Let \(D\) be an $n$-vertex edge-coloured digraph.
		If every two disjoint sets of vertices \(X,\, Y\) of size \(k\) satisfy
		\(
		\abs{\cols{(E(X,Y))}} \ge n,
		\)
		then \(D\) has a rainbow directed path of length at least \(n -2k\).
	\end{lemma}

	The next lemma can easily be proved using Chernoff's bound (cf.\ \Cref{thm:Chernoff}).
	\begin{lemma} \label{lemma:rdfs-cond-holds}
		Let \( 1/n \ll 1/C \ll \alpha \le 1/2\) and let \(D \sim \drgnp{n}{C/n}\) be uniformly coloured in \(\cols = [n]\).
		Then, with high probability, every two disjoint sets of vertices \(X,\, Y\) of size \(\alpha n\) satisfy
		\(
		\abs{\cols{(E(X,Y))}} \ge (1-\alpha)n.
		\)
	\end{lemma}

	Our main theorem now follows easily from the previous three lemmas.

	\begin{proof}[Proof of Theorem~\ref{main-thm}]
		Let $\eta$ and $\gamma$ satisfy
		\begin{equation*}
			1/n \ll 1/C \ll \eta \ll \gamma \ll \delta,
		\end{equation*}
		and write $V=V(D)$ and $\cols=[n]$.
		By \Cref{lemma:main} we may assume that there exists a subdigraph \(\Habs \subseteq D\) with \(\abs{V(\Habs)} \le \gamma n\)
		and the properties stated in \Cref{lemma:main}, so in particular \(\abs{\cols(\Habs)} =\abs{V(\Habs)} -1 \); and by \Cref{lemma:rdfs-cond-holds} (on input $\alpha=\eta/4$) that any two disjoint subsets \(X,Y\subseteq V\) of size 
		\(k=\eta n/4\)
		satisfy
		\(\abs{\cols(E(X,Y))} \ge (1-\eta/4)n= n - k\).
		Let $\cols_2 = \cols \setminus \cols(\Habs)$ and let
		\(V_2\) be an arbitrary subset of \(V\setminus V(\Habs)\) of size \(n-|V(\Habs)|-k\). 
		Then $\abs{\cols_2} \ge n - \gamma n$, $\abs{V_2} = \abs{\cols_2} - k - 1$, and, for any two disjoint subsets $X, Y \subseteq V_2$ of size $k$,
		\[
		\abs{\cols(E(X,Y)) \cap \cols_2} \ge |\cols_2| - k \ge \abs{V_2}.
		\]
		Then, by \Cref{lemma:rdfs-lemma}, applied to the subgraph of $D$ spanned by edges in $D[V_2]$ coloured in $\cols_2$, there is a rainbow directed path $P_2$ of length at least 
        \(\abs{V_2} - 2k\) with vertices in $V_2$ and colours in $\cols_2$.
		Let \(V_2' = V \setminus \left( V(\Habs) \cup V(P_2) \right)\) and \(\cols_2 ' = \cols \setminus \left( \cols(\Habs) \cup \cols(P_2) \right)\) be the set of vertices and colours used by neither $\Habs$ nor $P_2$.
		Observe that $|V_2'| \le 3k \le \eta n - 2$
		and \(\abs{\cols_2'} = \abs{V_2'}+2 \).
		Denote the first and last vertices of $P_2$ by $x$ and $y$.
		Then by the property of \(\Habs\) there exists a rainbow directed path \(Q\) from $y$ to $x$, with  
		\(V(Q)=V(\Habs) \cup V_2' \cup \{x,y\}\),
		and \(\cols(Q)=\cols(\Habs)\cup \cols_2'\).
		The concatenation of \(P_2\) and \(Q\) gives a rainbow directed Hamilton cycle, as desired.
	\end{proof}

    Finally we prove \Cref{undirected-main-thm}.
    \begin{proof}[Proof of \Cref{undirected-main-thm}]
        Let \(C\) be given by \Cref{main-thm} such that for any digraph \(D_0\) on \([n]\) with minimum semidegree at least \(\delta n/4\), the perturbed digraph
        \(D_0 \cup \dnp{n}{\frac{C}{2n}}\) uniformly coloured in \([n]\) has with high probability a rainbow directed Hamilton cycle.
        Let \(N=\binom{n}{2}\) and \(e_1, \hdots, e_N\) be an enumeration of the edges of the (undirected) complete graph on \([n]\).
        For each \(0\le i \le N\), 
		define a randomly edge-coloured digraph \(\Gamma_i\)
        on \([n]\) as follows, recalling that $G_0$ is an $n$-vertex graph with minimum degree at least $\delta n$, and writing $e_j = \{x,y\}$.
        \begin{enumerate}[a)]
            \item For \(j>i\):
            \begin{enumerate}[1.]
                \item If \(e_j \in E(G_0)\), then add both \(xy, yx\) to \(E(\Gamma_i)\), and colour both edges with the same colour, chosen uniformly in \([n]\).
                \item If \(e_j \notin E(G_0)\), then add both \(xy, yx\) to \(E(\Gamma_i)\) together with probability \(\frac{C}{n}\), and colour both edges with the same colour, chosen uniformly in \([n]\).
            \end{enumerate}
           \item  For \(j\le i\):
           \begin{enumerate}[1.]
               \item if \(e_j \in E(G_0)\),
               toss a coin that comes up heads with probability 1/3.
               If it comes up heads, add both \(xy, yx\) to \(E(\Gamma_i)\)
               and colour each independently and uniformly at random in \([n]\).
               If it comes up tails, add each \(xy, yx\) independently with probability \(\frac{C}{2n}\) to \(E(\Gamma_i)\), and colour each independently and uniformly at random in \([n]\).
               \item If \(e_j \notin E(G_0)\), add each of \(xy, yx\) independently with probability \(\frac{C}{2n}\)  to \(E(\Gamma_i)\), and colour each independently and uniformly at random in \([n]\).
           \end{enumerate}
        \end{enumerate}

		The sequence is coupled, i.e.\ \(\Gamma_{i-1}\) and \(\Gamma_{i}\) agree as probability spaces (but not as digraphs) on all edges apart from \(e_i\), for $i \in [N]$.
        Clearly \(\Gamma_0 \sim G_0 \cup \gnp{n}{C/n} \) and is uniformly coloured in \([n]\), where we view the undirected edges as two parallel directed edges.
        It is also easy to see that \(\Gamma_N \sim D_0 \cup \dnp{n}{\frac{C}{2n}}\) and it is uniformly coloured in \([n]\), where \(D_0\) is a random subgraph of \(G_0\), with each edge selected independently (as two parallel edges) with probability 1/3.
        Indeed, for \(e=\{x,y\} \notin E(G_0)\), clearly each of \(xy, yx\) is in \(E(\Gamma_N)\) independently with probability \(\frac{C}{2n}\).
        For \(e\in E(G_0)\), with probability \(1/3\), both orientations of \(e\) are added to \(E(D_0)\), and each is coloured independently and uniformly in \([n]\);
        otherwise each orientation is added independently to \(E(\Gamma_0)\) with probability \(\frac{C}{2n}\) and coloured independently and uniformly in \([n]\).
		A straightforward application of the union bound and Chernoff's bound implies that with high probability the minimum semidegree of \(D_0\) is at least \(\delta n /4\).
        Hence, by the choice of \(C\), with high probability \(\Gamma_N\) has a directed rainbow Hamilton cycle.
        We will show that for each \(i\in [N]\)
		\begin{equation} \label{eqn:Gamma}
			\prob{\Gamma_{i-1} \text{ has a rainbow  Hamilton cycle} } \ge \prob{\Gamma_i \text{ has a rainbow  Hamilton cycle}},
		\end{equation}
        which then implies the theorem.
        
        To this end, reveal the randomness on all edges apart from \(e_i = xy\), i.e.\ on all those edges that \(\Gamma_{i-1}, \Gamma_i\) automatically agree on.
        There are three possibilities:
        \begin{enumerate}
            \item 
				\(\Gamma_{i-1} \cup \{xy, yx\} \) has no rainbow Hamilton cycle, irrespective of the colouring of \(xy, yx\). Then, regardless of the outcome of $e_i$, both $\Gamma_{i-1}$ and $\Gamma_i$ have no rainbow Hamilton cycle.
            \item 
				\(\Gamma_{i-1}\) has a rainbow Hamilton cycle without using \(\{xy, yx\} \).
				Then, regardless of the outcome of $e_i$, both $\Gamma_{i-1}$ and $\Gamma_i$ have a rainbow Hamilton cycle.
            \item 
				Both events above do not hold. In other words, \(\Gamma_{i-1}\) has a rainbow Hamilton cycle using one of \(xy, yx\), for some choice of colours, but it has no rainbow Hamilton cycle that avoids both of $xy, yx$.
				Let \(\cols_1, \cols_2 \subseteq \cols\), satisfy the following:
				\(\Gamma_{i-1} \cup xy\) has a rainbow Hamilton cycle through $xy$ if and only if \(\cols(xy) \in \cols_1\), and \(\Gamma_{i-1} \cup yx\) has a rainbow Hamilton cycle through $yx$ if and only if \(\cols(yx) \in \cols_2\).
				(So $\cols_1 \cup \cols_2 \neq \emptyset$.)
        \end{enumerate}
        Conditioning on either of the first two cases, both \(\Gamma_{i-1}, \Gamma_i\) have the same probability to have a rainbow Hamilton cycle.
        Conditioning on the third case, 
        the probability that \(\Gamma_{i-1}\) has a rainbow Hamilton cycle is
		\begin{equation*}
			\left\{
				\begin{array}{ll}
					\frac{|\cols_1 \cup \cols_2|}{n} & \text{ if $e_i \in E(G_0)$}, \\
					\frac{C}{n}\cdot \frac{\abs{\cols_1\cup \cols_2}}{n} & \text{ if $e_i \notin E(G_0)$}.
				\end{array}
			\right.
		\end{equation*}
		and the probability that $\Gamma_i$ has a rainbow Hamilton cycle is
		\begin{equation*}
			\left\{
				\begin{array}{ll}
					\frac{1}{3} \cdot \left(1 - \left(1 - \frac{|\cols_1|}{n}\right) \left(1 - \frac{|\cols_2|}{n}\right)\right) 
					+ \frac{2}{3} \cdot \left(\frac{C}{2n} \cdot \frac{|\cols_1| + |\cols_2|}{n} - \frac{C^2}{4n^2} \cdot \frac{|\cols_1||\cols_2|}{n^2}\right)
					\le \frac{|\cols_1| + |\cols_2|}{2n}
					\le \frac{|\cols_1 \cup \cols_2|}{n}
					& \text{ if $e_i \in E(G_0)$}, \\[.8em]
					\frac{C}{2n}\cdot \frac{\abs{\cols_1}+ \abs{\cols_2}}{n}
					- \frac{C^2}{4n^2} \cdot  \frac{\abs{\cols_1} \abs{\cols_2} }{n^2}
					\le
					\frac{C}{n}\cdot \frac{\abs{\cols_1}+ \abs{\cols_2}}{2n}
					\le \frac{C}{n} \cdot \frac{|\cols_1 \cup \cols_2|}{n}
					& \text{ if $e_i \notin E(G_0)$}.
				\end{array}
			\right.
		\end{equation*}
		This shows that in either case the probability that \(\Gamma_{i-1}\) has a rainbow Hamilton cycle is at least as large as the probability that $\Gamma_i$ has one, thereby proving \eqref{eqn:Gamma} and thus the theorem.
    \end{proof}

\section{Preliminaries} \label{section:prelims}

	Next we collect three preliminary results that we need: the Chernoff bound, \Cref{thm:Chernoff};
	that random sparse subgraph of dense hypergraphs have large matchings,  \Cref{lemma:linear-matching};
	and that in the perturbed digraph, between any two vertices, there is a large rainbow collection of directed paths of length three,  \Cref{lemma:con}.

	\begin{theorem}[{Chernoff Bound, \cite[eq.\ (2.8) and Theorem 2.8]{jlr}}] \label{thm:Chernoff} 
		For every \(\eps>0\) there exists \(c_\eps>0\) such that the following holds.
		Let \(X\) be the sum of mutually independent indicator random variables and write \(\mu = \expect{X}\).
		Then
		\[
		\prob{\abs{X- \mu} \ge \eps \mu} \le 2\exp(-c_\eps\, \mu).
		\]
	\end{theorem}

	The next lemma, despite its technical appearance, proves the following straightforward statement: quite sparse random subgraphs of dense hypergraphs contain, with very high probability, a matching of linear size. A \emph{matching} in a hypergraph is a collection of pairwise vertex-disjoint edges. The \emph{degree} of a vertex \(v\) is the number of edges incident to \(v\).
	\begin{lemma} \label{lemma:linear-matching}
		Let \(1/n \ll \rho \ll \alpha, c, 1/r\) where $r \ge 2$ is an integer.
		Let \(\calH\) be an \(r\)-uniform hypergraph on \(n\) vertices with at least \(\alpha n^r\) edges.

		Writing $m = cn$, let \(\calH_m\) be the random subgraph of \(\calH\)
		that consists of \(m\) edges of \(\calH\), chosen with replacement and uniformly at random.
		Then, with probability at least \(1-\exp\left(-\frac{c \alpha^2 n}{3}\right)\),
		the hypergraph \(\calH_m\) has a matching of size at least \(\rho n\).

		Writing $p = cn^{-r+1}$, let \(\calH_p\) be the random subgraph of \(\calH\) where we keep each edge independently with probability 
		\(p\).
		Then, with probability at least
		\(1-\exp\left(-\frac{c \alpha^2 n}{3r}\right)\), 
		the hypergraph \(\calH_p\) has a matching of size at least \(\rho n\).
	\end{lemma}
	
	\begin{proof}
		Write \(\beta(\calG)\) for the size of a largest matching of a hypergraph \(\calG\).

		It is not hard to see that \(\calH\) contains an induced subgraph \(\calH'\) of minimum degree at least \(\alpha n^{r-1}\) on at least \(\alpha n\) vertices.

		We first prove the result for \(\calH_m\).
        Let \(\calH'_m\) be the hypergraph with vertices \(V(\calH')\) and edges 
        \(E(\calH_m) \cap E(\calH')\).
        We will show that, with high probability, \(\calH'_m\) has a matching of size at least \(\rho n\). Clearly \(\beta(\calH_m) \ge \beta(\calH'_m)\), so the lemma readily follows.
        
		Suppose \(\beta(\calH'_m) < \rho n\), and let \(M\) be a maximal matching in \(\calH'_m\).
		Then \(S = V(\calH')\setminus V(M)\) is an independent set in \(\calH'_m\) and
		\(\abs{S} \ge (\alpha-r\rho)n\).
		By the minimum degree condition of \(\calH'\), the number of edges with all vertices in \(S\) is at least
		\(
			\frac{1}{r}|S| (\alpha n^{r-1} - r\rho n^{r-1}) \ge
			\frac{1}{r} (\alpha-r \rho)\, (\alpha - r \rho) n^r
			\ge 
			\frac{\alpha^2}{2r}\, n^r.
		\)
		This gives
		\begin{align*}
			\prob{S \text{ is independent in } \calH'_m}
			= \left(1- \frac{e(\calH[S])}{e(\calH)}\right)^m
			&\le \exp \left( - cn \cdot \frac{\frac{\alpha^2}{2r} \, n^r}{\binom{n}{r}} \right) \\
			&\le 
			\exp \left( - \frac{\alpha^2}{2}(r-1)!\, c n \right)\, .
		\end{align*}
		Then, since \(\alpha,r \rho<1/2\), the number of \(S\subseteq  V(\calH)\) with \( \abs{S} \ge (\alpha-r\rho)n\) is at most
		\[
			n\binom{n}{(\alpha-r\rho)n} 
			\le n\binom{n}{r\rho n}
			\le n \cdot \left(\frac{\e}{r\rho}\right)^{r\rho n}
			\le \e^{2r\rho n} (r\rho)^{-r\rho n} 
			= \exp\big((2r\rho - r\rho \log (r\rho))n\big).
		\]
		Thus, by the union bound,
		\(
		\prob{\beta(\calH'_m) < \rho n}
		\le \exp \left(- f_{c,r,\alpha} (\rho) n \right),
		\)
		where
		\[
		f_{c,r,\alpha}(\rho) = - 2r\rho + r\rho \log (r\rho) + \frac{(r-1)!}{2} c \alpha^2\, .
		\]
		Since \(f_{c,r,\alpha} (\rho)\) is continuous near \(0\) and \(
		f_{c,r,\alpha} (\rho) \tendsto \frac{(r-1)!}{2}\, c \alpha^2
		\)
		as \(\rho \tendsto 0\), for \(\rho\) sufficiently small 
		\(f_{c,r,\alpha} (\rho) \ge \frac{(r-1)!}{3}\, c \alpha^2 \ge \frac{c \alpha^2}{3} \), 
		which gives the first part of the lemma.

		For the second part of the lemma observe that the same argument works: with \(S\) as above, in \(\calH_p\) we have
		\[
		\prob{S \text{ is independent in } \calH_p}
		= (1-p)^{e(\calH[S])}
		\le 
        \e^{-p e(\calH[S])}
        \le 
		\exp \left( - \frac{c \alpha^2 n}{2r} \right)
		\]
		and a similar calculation as above shows that, for \(\rho\) sufficiently small,  the probability there is such an \(S\) is at most \(\exp\left(-\frac{c\alpha^2 n}{3r}\right)\).
	\end{proof}
 
	\begin{lemma}[Triangles and short paths] \label{lemma:con}
		Let \(1/n \ll 1/C \ll \lambda \ll \pathb \ll \delta, q\).
		Let \(\cols\) be a set of colours of size \(qn\),
        let $D_0$ be a digraph on \(n\) vertices with minimum semidegree at least \(\delta n\),
		and suppose that \(D \sim D_0\cup \dnp{n}{C/n}\) is uniformly coloured in \(\cols\).
		Then, with probability at least \(1-\exp\left(-\lambda\, n \right)\), the following holds.
		For any two vertices \(u,v \in V(D)\)
		there is a matching \(M\) of size  at least 
		\(\rho n\)
		such that  \(\bigcup_{xy \in M} \{ux,xy,yv\}  \) is rainbow.
  
		Moreover, with probability at least \(1-\exp\left(-\lambda n \right)\), for any $u \in V(D)$
		there is a matching \(M\) of size at least \(\rho n\)
		such that  \( \bigcup_{xy \in M} \{ux,xy,yu\}  \) is rainbow.
  	\end{lemma}

	\begin{proof}
		Let $\rho_1$ satisfy
		\( 1/C \ll \lambda \ll \pathb \ll \pathb_1 \ll \delta,q \).
		Fix \(u,v \in V\).
		By the minimum semidegree assumption on $D_0$, there exist disjoint subsets
		\(N_u \subseteq N_{D_0}^+(u)\),
		\(N_v \subseteq N_{D_0}^-(v)\),
		each of size \(\delta n/2\).
		Consider the oriented bipartite graph with bipartition \((N_u, N_v)\) and edges
		\[
		\left\{\, zw \in E(\dnp{n}{C/n}):\: z\in N_u,\ w \in N_v \right\}\,
		\]
		Then, by \Cref{lemma:linear-matching} (applied with $(\alpha, c, r, n)_{\ref{lemma:linear-matching}} = (1/4, \delta C, 2, \delta n)$), with probability 
		\(1 - \e^{-\Omega( \delta C n)}\), 
		it has a matching \(M\) of size \(\pathb_1 n\).
		For each \(zw\in M\), reveal whether the directed path \(uzwv\) is rainbow, without exposing the colours.
		Then each \(uzwv\) is rainbow independently with probability \(1-o(1)\). 
		Hence by Chernoff's bound (\Cref{thm:Chernoff}), with probability \(1-\e^{-\Omega(\pathb_1 n)}\),
		there is \(M'\subseteq M\) with \(\abs{M'} \ge \pathb_1 n /2\) such that for each \(zw \in M'\), \(uzwv\) is rainbow.
        Let \( \calP' = \{uzwv: zw \in M'\}\), and
		observe that the same colour may still repeat on different paths in $\calP'$.

		We now show that we can find a large subset of \(\calP'\) where the paths are pairwise colour-disjoint.
		Reveal the colours on the edges of the paths in \(\calP'\).
		By symmetry, each triple of distinct colours in \(\cols\) is equally likely to appear in \(\calP'\).
		Hence \(\calP'\) corresponds to selecting uniformly at random with replacement \(\abs{\calP'} \ge \pathb_1 n/2\) edges from the complete 3-graph with vertex set \(\cols\).
		Thus, by \Cref{lemma:linear-matching} (applied with $(\alpha,c,r,n)_{\ref{lemma:linear-matching}} = (1/7,\rho_1/(2q),3,qn)$), 
		with probability \( 1- \e^{-\Omega(\pathb_1 n)}\), this $3$-graph has a matching of size $\rho n$.
		This corresponds to an \(M\dprm\subseteq M'\) of size \(\pathb n\)
		so that \( \bigcup_{xy \in M\dprm} \{ ux, xy, yv\}\) is rainbow.

		The probability this fails for some pair 
		\(u,v\) is, by the union bound, at most
		\[
		n^2 \cdot \left( \e^{-\Omega( \delta C n)} + \e^{-\Omega(\pathb_1 n)} + e^{-\Omega( \pathb_1 n)}\right)
		\le 
		\e^{-\lambda n},
		\]
		proving the first statement of the lemma.

		The second statement of the lemma follows similarly by finding a large matching between disjoint subsets of \(N^-_{D_0}(u)\) and \(N^+_{D_0}(u)\).
	\end{proof}
\section{Finding absorbers} \label{section:absorber}

	In this section we show how to find `absorbers', which are the building blocks for the digraph \(\Habs\) in \Cref{lemma:main}.

	\begin{definition}[Absorber] \label{def:absorber}
		Let \(v\) be a vertex and \(c\) a colour.
		A \emph{\((v,c)\)-absorber}
		is an edge-coloured digraph $A_{v,c}$ with \(v\in V(A_{v,c})\) and \(c\in \cols(A_{v,c})\) that has two directed paths \(P, P'\) with the following properties:

		\begin{itemize}
			\item they are rainbow;
			\item they have the same first and last vertex;
		\item 
		\(P\) is spanning in \(V(A_{v,c})\) and
		\( V(P') = V(P)\setminus \{v\} =  V(A_{v,c})\setminus \{v\} \);
		\item 
		\(P\) is spanning in \(\cols(A_{v,c})\) and
		\(
		\cols(P')
		= \cols(P)\setminus\{c\}
		= \cols(A_{v,c})\setminus \{c\}
		\).
		\end{itemize}
		We call \(P\) the \emph{\((v,c)\)-absorbing path}
		and \(P'\) the \emph{\((v,c)\)-avoiding path}.
		The \emph{internal vertices} of \(A_{v,c}\) are \(V(A_{v,c})\setminus \{v\}\) and the \emph{internal colours} are \(\cols(A_{v,c}) \setminus \{c\}\).
		The \emph{first and last vertices} of the absorber are the first and last vertices of $P$ and $P'$.
	\end{definition}

    The aim of this section is to prove the following lemma, which says that for any vertex \(v\), colour \(c\) and any small (but linear in size) sets of forbidden vertices and colours, we can find a \((v,c)\)-absorber.

	\begin{lemma} \label{lemma:absorbers-exist}
		Let \(1/n \ll 1/C \ll \gadgetb \ll \delta <1/2\).
		Let $D_0$ be a digraph on \(n\) vertices with minimum semidegree at least \(\delta n\), and suppose that \(D\sim D_0 \cup \dnp{n}{C/n}\) is uniformly coloured in \(\cols=[n]\).
		Then, with high probability,
		the following holds.
		For any \(v\in V(D)\) and \(c \in \cols\),
		and for all \(V' \subseteq V(D)\) and \(\cols' \subseteq \cols\)
		each of size at least \((1-\gadgetb )n\),
		there exists a \((v,c)\)-absorber on  13 vertices with internal vertices in \(V'\) and internal colours in \(\cols'\).
	\end{lemma}

	Our absorbers are depicted in \Cref{fig:absorber}.
	We describe their structure here, though it might be easier to read off the structure from the figure.
	They consist of vertices $\{v,v_1,v_2\}$ that induce an oriented triangle with edges $\{v_1v, vv_2, v_1v_2\}$, and vertices $\{x,y,z,u,w_1,w_2\}$ inducing an oriented $K_{2,4}$ with edges $\{xy,xz,yu,zu,w_2y,w_2z,yw_1,zw_1\}$, and two directed paths of length $3$ --- $P_1$ from $v_2$ to $x$ and $P_2$ from $w_1$ to $w_2$ --- whose interiors are vertex-disjoint and disjoint of previously mentioned vertices.
	The absorber is equipped with an edge colouring satisfying the following: $\cols(yu) = c$, $\cols(v_1v) = \cols(xy)$, $\cols(vv_2) = \cols(zu)$, $\cols(v_1v_2) = \cols(xz)$, $\cols(w_2y) = \cols(w_2z)$, $\cols(yw_1) = \cols(zw_1)$, and the colours of the edges of \(P_1\) and \(P_2\) are different from one another and from those on the rest of the absorber.
	The $(v,c)$-absorbing path $P$ and $(v,c)$-avoiding path $P'$ are  
	\begin{equation*}
		P = v_1vv_2P_1xzw_1P_2w_2yu,
		\qquad
		P' = v_1v_2P_1xyw_1P_2w_2zu.
	\end{equation*}

	\begin{figure}[ht] 
		\centering
		\begin{subfigure}{0.7\linewidth}
				\includegraphics[width=\linewidth]{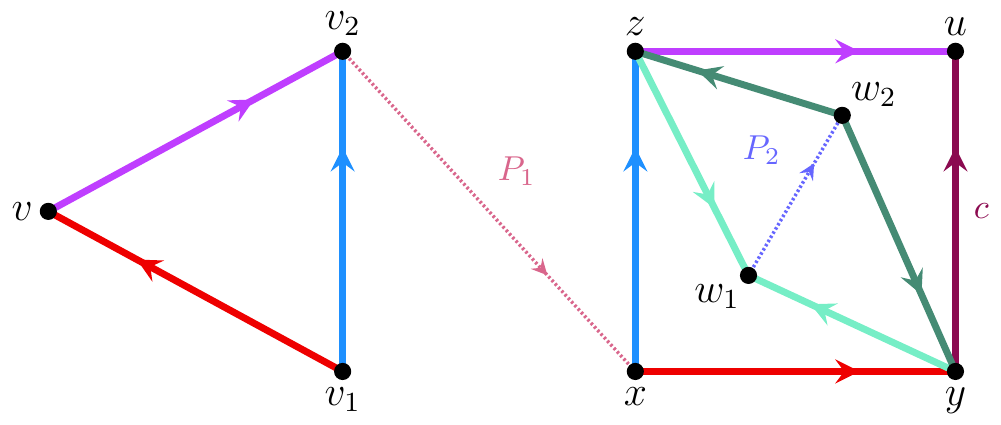}
			\end{subfigure}%
		\vskip 5pt
		\begin{subfigure}[b]{0.45\linewidth}
			\includegraphics[width=\linewidth]{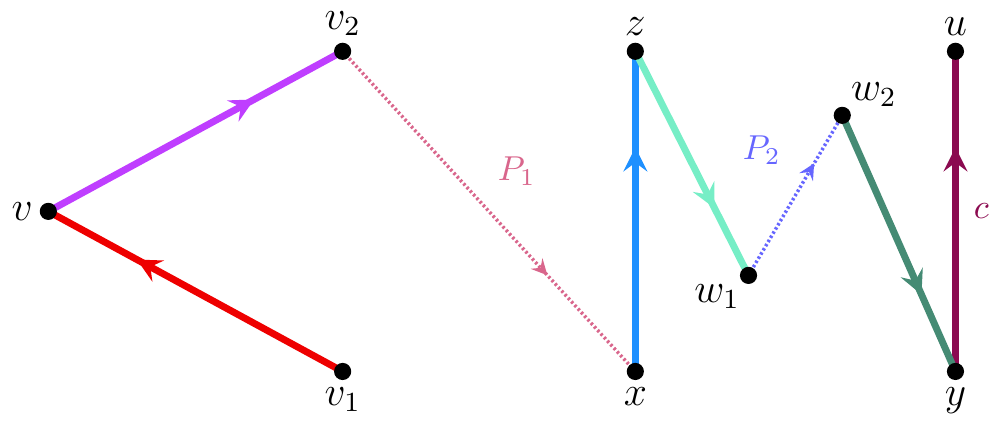}
			\end{subfigure}
		\hfill
		\begin{subfigure}[b]{0.45\linewidth}
		\includegraphics[width=\linewidth]{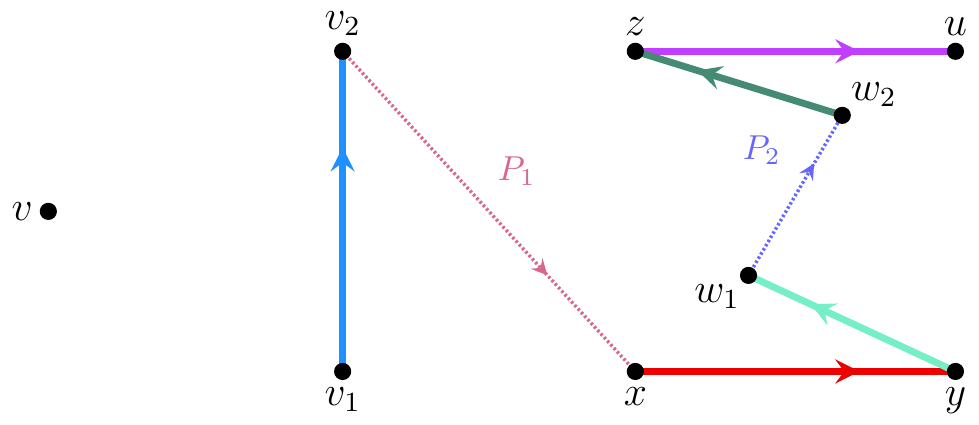}
			\end{subfigure}%
		\caption{At the top is the \((v,c)\)-absorber used in \Cref{lemma:absorbers-exist}. 
		At the bottom the first figure shows the \((v,c)\)-absorbing path and the second figure the \((v,c)\)-avoiding path. 
		The directed paths \(P_1, P_2\) have length 3 and are rainbow with colours disjoint of one another and of the other colours in the figure.}
		\label{fig:absorber}
	\end{figure}

	For the existence of the oriented triangle (on vertices $\{v,v_1,v_2\}$) and the directed paths of length three $P_1$ and $P_2$ we use \Cref{lemma:con}.
	The existence of the oriented $K_{2,4}$ (with the required colouring) is more involved and uses the regularity lemma for digraphs.
	This is the most technical part of our proof and is accomplished in \Cref{lemma:K24}.

	\subsection{Finding oriented squares with diagonal directed paths}

	The aim of this section is to prove \Cref{lemma:K24}, which allows us to find a \(K_{2,4}\) as required in \Cref{fig:absorber}.

	\subsubsection{Regularity preliminaries}
		
		The (directed) \emph{densities} of a pair of non-empty disjoint vertex sets \(V,W\) in a digraph are 
		\( d^+(V,W)  = \frac{e^+(V,W)}{\abs{V}\abs{W}}\)
		and
		\( d^-(V,W)  = \frac{e^-(V,W)}{\abs{V}\abs{W}}\).
		For simplicity $d(V,W)$ will always stand for $d^+(V,W)$.
		
		\begin{definition}
			An ordered pair of disjoint vertex sets \((V,W)\) is \emph{\(\eps\)-regular} if for any \(V' \subseteq V\), \(W' \subseteq W\) with
		   \(\abs{V'} \ge \eps \abs{V}, \abs{W'} \ge \eps \abs{W} \)
		   we have
		   \(\abs{d(V',W') - d(V,W)} < \eps\).
		\end{definition}

		Since we need regular pairs to be sufficiently dense, we will find the following definition useful.
		\begin{definition} \label{def:regularity}
		   An ordered pair of disjoint vertex sets \((V,W)\) is 
		   \emph{\((\eps,\rho)\)-super-regular} if it is \(\eps\)-regular and for every \(v\in V\), \(\deg^+(v,W) \ge \rho \abs{W}\) and for every \(w\in W\), \(\deg^-(w,V) \ge \rho \abs{V}\).
		\end{definition}

		We state without proof the following straightforward consequences of \Cref{def:regularity}.

		\begin{lemma} \label{lemma:reg:few-bad}
			 If \((V,W)\) is \(\eps\)-regular, then for all but at most \(\eps \abs{V}\) vertices  \(v \in V\) we have \(\deg^+ (v, W) \ge (d^+(V,W)-\eps) \abs{W}\), and for all but at most \(\eps \abs{W}\) vertices \(w\in W\) we have \(\deg^- (w, V) \ge (d^-(V,W)-\eps) \abs{V}\).
		\end{lemma}

		\begin{lemma} \label{lemma:reg:edge-sampling}
			Let \(1/n\ll \eps,\rho, p\).
			Let \((V,W)\) be \((\eps,\rho)\)-super-regular with $|V| = |W| = n$. 
			Sample each edge independently with probability \(p\).
			Then, with probability at least \(1-\exp(-\Omega(p\rho n))\), the resulting pair is \((\eps,p \rho/2)\)-super-regular.
		\end{lemma}

		We use the following degree form of the regularity lemma for digraphs, from \cite{kuhn-osthus-diregularity-stmt}. The original version of Szemer\'edi's regularity lemma~\cite{szemeredi} for digraphs was first proved by Alon and Shapira~\cite{alon-shapira-di-regularity}.
        The version we use is Lemma 7 in~\cite{kuhn-osthus-diregularity-stmt}.

		\begin{theorem}[Degree form of the regularity lemma for digraphs] \label{lemma:regularity}
			For every \(\eps \in (0,1)\) and every integer \(M'\), there are integers \(M\) and \(n_0\) such that the following holds.
			If \(D\) is a digraph on \(n \ge n_0\) vertices and \(d \in [0,1]\) is any real number, then there is a partition \((V_0,\hdots, V_k)\) of \(V(D)\), a spanning subgraph \(D'\) of \(D\) and a set \(U\) of ordered pairs \((i,j)\),  \(1\le i,j \le k\) and \(i\neq j\), such that the following  hold:
			\begin{itemize}
				\item \(M' \le k \le M\);
				\item \(\abs{V_0} \le \eps n\);
				\item \(\abs{V_1} = \cdots = \abs{V_k}\);
				\item for all  \(x\in V(D)\),
                \(\deg_{D'}^+(x) > \deg_D^+(x) - (d+\eps)n\)
                and
                \(\deg_{D'}^-(x) > \deg_D^-(x) - (d+\eps)n\);
				\item \(\abs{U} \le \eps k^2\);
				\item for every ordered pair \((i,j) \notin U\) with \(i\neq j\) the bipartite digraph \(D[V_i, V_j]\) is \(\eps\)-regular;
				\item \(D'\) is obtained from \(D\) by deleting the following edges of \(D\): all edges with both endpoints in \(V_i\), for all \(i\in [k]\); all edges \(E_D(V_i, V_j)\) for all \((i,j) \in U\); and all edges \(E_D(V_i, V_j)\) for all \((i,j) \notin U\), \(i\neq j\), with \(d_D(V_i,V_j) < d\).
			\end{itemize}
		\end{theorem}

		We call $V_1, \dots, V_k$ the \emph{clusters} of the partition and $V_0$ the \emph{exceptional set}. 
		Note that the last two conditions of the lemma imply that for all $1 \le i, j \le k$ with $i \neq j$, the bipartite
		graph $D'[V_i, V_j]$ is $\eps$-regular and has density either $0$ or at least $d$.

	\subsubsection{Finding gadgets}
		In the following lemma we show that, if $(A,B)$ and $(C,B)$ are two dense regular pairs in a digraph whose edges are uniformly coloured, and given a dense bipartite graph $H$ with bipartition $(A,C)$, with very high probability there is a dense subgraph $H' \subseteq H$, such that for every edge $ac$ in $H'$ there is an oriented path of length $2$ from $a$ to $c$ through $B$ which is monochromatic (and of course both edges are directed outwards from \(a\) and \(c\)).

		By applying this lemma twice (reversing edges for the second application, see the proof of \Cref{lemma:K24}), given four dense regular pairs $(W_2,Y), (W_2,Z), (Y,W_1), (Z,W_1)$ in a digraph whose edges are uniformly coloured, with very high probability we can find $(y,z,w_1,w_2) \in Y \times Z \times W_1 \times W_2$ such that $\cols(w_2y) = \cols(w_2z) \neq \cols(yw_1) = \cols(zw_1)$,\footnote{In fact, we can find quadratically many such quadruples, with distinct pairs $(y,z)$.} in line with the edges spanned by $\{y,z,w_1,w_2\}$ in our gadgets.

		\begin{lemma} \label{lemma:ABC}
			Let $1/n \ll \gamma, \eps, \lambda  \ll \alpha, \rho, 1/q$. 
			Let $D$ be a digraph, and let $A, B, C$ be disjoint sets of vertices in $D$ of size $n$, such that $(A,B), (C,B)$ are $\eps$-regular and have density at least $\rho$.
			Let $H$ be a digraph with $E(H) \subseteq A \times C$ and $e(H) \ge \alpha n^2$.
			Assign each edge in $D$ a colour in $\cols = [qn]$ uniformly and independently.
			Then, 
            with probability at least $1 - \exp(-\lambda n)$,
            the number of edges $ac$ in $H$ for which there exists $b \in B$ such that $ab,cb \in E(D)$ and $\cols(ab) = \cols(cb)$, is at least 
            $\gamma n^2$.
		\end{lemma}

		\begin{proof}
			Let $\beta$ be a constant satisfying $\gamma, \lambda \ll \beta \ll \alpha, \rho, 1/q$.

			Let $H_1$ be the subgraph of $H$ spanned by edges $ac$ such that $|\Np(a) \cap \Np(c) \cap B| \ge \rho^2 n / 2$.
			We claim that $e(H_1) \ge \alpha n^2/2$.
			Indeed, let $A' = \{a \in A : |\Np(a) \cap B| \ge (\rho - \eps)n\}$. Then $|A'| \ge (1 - \eps)n$ by $\eps$-regularity. Given $a \in A'$, we have that $|\Np(a) \cap \Np(c) \cap B| \ge (\rho - \eps)^2n \ge \rho^2 n / 2$ for all but at most $\eps n$ vertices $c \in C$, using $\eps \ll \rho$ for the last inequality. Thus $e(H_1) \ge e(H) - |A \setminus A'| \cdot n - |A'| \cdot \eps n \ge e(H) - 2\eps n^2 \ge \alpha n^2 / 2$, using $\eps \ll \alpha$.

			Let $M_1, \ldots, M_k$ be a maximal collection of pairwise edge-disjoint matchings in $H_1$ of size at least $\alpha n / 4$. We claim that $k \ge \alpha n / 4$.
			Indeed, otherwise, since $|M_i| \le n$ for every $i$, we have $e(H_1 \setminus (M_1 \cup \ldots \cup M_k)) \ge \alpha n^2 / 2 - \alpha n^2 / 4 = \alpha n^2 / 4$. But the graph $H_1 \setminus (M_1 \cup \ldots \cup M_k)$ can be decomposed into at most $n$ matchings, showing that there is a matching $M_{k+1}$ in $H \setminus (M_1 \cup \ldots \cup M_k)$ of size at least $\alpha n / 4$, a contradiction.

			Let $M_i'$ be the (random) submatching $\{ac \in M_i : \text{there is $b \in B$ such that $\cols(ab) = \cols(cb)$}\}$.
			For $ac \in M_i$, we have 
			\begin{align*}
				\prob{ac \in M_i'} 
				& = 1 - \left(1 - \frac{1}{qn}\right)^{|\Np(a) \,\cap\, \Np(c) \,\cap\, B|} \\
				& \ge 1 - \exp\left(-\frac{|\Np(a) \cap \Np(c) \cap B|}{qn}\right) \\
				& \ge 1 - \exp\left(-\frac{\rho^2}{2q}\right)
                = 8 \alpha^{-1} \beta.
			\end{align*}
            Thus, $\expect{|M_i'|} \ge |M_i| \cdot 8 \alpha^{-1} \beta \ge 2 \beta n$.
			Because, for each $i \in [k]$, the events $\{ac \in M_i'\}$ with $ac \in M_i$ are independent, by Chernoff's bound we have
            \begin{align*}
				\prob{|M_i'| \le \beta n} \le 
				\prob{|M_i'| \le \frac{\expect{|M_i'|)}}{2}} 
				\le \exp\left(-\Omega\left(\expect{|M_i'|}\right)\right)
				\le \exp(-\Omega(\beta n)).
			\end{align*}
			By a union bound, with probability at least $1 - 2n \exp(-\Omega(\beta n)) \ge 1 - \exp(-\lambda n))$, we have $|M_i'| \ge \beta n$ for every $i \in [k]$, showing $|M_1' \cup \ldots \cup M_k'| \ge (\alpha n / 4) \cdot \beta n = \gamma n^2$. This proves the lemma.
		\end{proof}

		In the next lemma we construct the part spanned by $\{x,y,z,u\}$. Recall that we would like the colours of $zu,xz,xy$ to match the colours of the edges $vv_2,v_1v_2,v_1v$ in the oriented triangle in the absorber. To achieve this, we let $\cols_0$ be a collection of $\Omega(n)$ triples of distinct colours, which are pairwise disjoint, and ensure that the colours of $zu,xz,xy$ match the colours of one of the triples in $\cols_0$.

		\begin{lemma} \label{lemma:XYZU}
			Let $1/n \ll \lambda \ll \eps \ll \alpha,\rho,q_0 \le 1$ and let $1/n \ll \mu \le 1$.
			Let $X, Y, Z, U$ be disjoint sets of $\mu n$ vertices in a digraph $D$ such that $(X,Y), (X, Z), (Y, U), (Z, U)$ are $(\eps, \rho)$-super-regular.
			Let $H$ be a digraph with $E(H) \subseteq Y \times Z$ and $e(H) \ge \alpha \mu^2 n^2$.
			Let $\cols = [n]$, let $c \in \cols$, and let $\cols_0$ be a set of pairwise disjoint triples of colours in $\cols \setminus \{c\}$ of size at least $q_0n$.

			Assign each edge in $D$ a colour from $\cols$, uniformly and independently.
			Then, with probability at least $1 - \exp(-\lambda n)$, there is a quadruple $(x,y,z,u) \in X \times Y \times Z \times U$ with $xy, xz, yu, zu \in E(D)$ and $yz \in E(H)$, such that $(\cols(zu), \cols(xz), \cols(xy)) \in \cols_0$ and $\cols(yu) = c$.
		\end{lemma}
		\begin{proof}
			Let $\gamma, \gamma_1, \gamma_3, \lambda'$ satisfy
			\begin{equation*}
				1/n \ll \lambda \ll \lambda' \ll \gamma_3 \ll \gamma_1 \ll \gamma \ll \eps, \mu.
			\end{equation*}
			Let $D_0$ be the subgraph of $D$ consisting of edges in $D$ contained in $X \times (Y \cup Z)$ or in $(Y \cup Z) \times U$.
			We find a quadruple $(x,y,z,u)$ satisfying the requirements of the lemma, by revealing the colours of $E(D_0)$ in the order \( E(Y,U), E(Z,U), E(X,Y), E(X,Z)\).
			At each step we find a linear size monochromatic matching between the corresponding vertex sets.
			To ensure that each new matching extends nicely previously chosen matchings, we will `clean up' the graph before finding it.

			First, let $D_1$ be the subdigraph of $D_0$ obtained by removing edges $yu \in Y \times U$ with $|\Nm_{D_0}(u) \cap \Np_H(y)| \le (\alpha \rho \mu / 4)n$. 
			We claim that $e(D_1[Y,U]) \ge (\alpha\rho \mu^2 / 4) n^2$.
			Indeed, writing $Y' = \{y \in Y: \dpp_H(y) \ge (\alpha/2)|Z|\}$, observe that $|Y'| \ge (\alpha/2)|Y|$ (otherwise $e(H) \le (|Y \setminus Y'| \cdot (\alpha/2)|Z| + |Y'|\cdot |Z| < |Y| \cdot (\alpha/2)|Z| + (\alpha/2)|Y| \cdot |Z| = \alpha |Y||Z| = \alpha \mu^2 n^2$, a contradiction).
			For every $y \in Y'$, all but at most $\eps |U|$ vertices $u$ in $U$ have at least $(\rho - \eps) |\Np_H(y)| \ge (\alpha \rho \mu / 4)n$ in-neighbours in $\Np_H(y)$. This implies that $e(D_1[Y,U]) \ge e(D_0[Y',U]) - |Y'| \cdot \eps |U| \ge (\rho - 2\eps)|Y'||U| \ge (\rho/2) \cdot (\alpha/2) \cdot \mu^2 n^2 = (\alpha \rho \mu^2 / 4) n^2$.

			Reveal the colours of \(E_{D_1}(Y,U)\).
			Since $e(D_1[Y,U]) \ge (\alpha \rho \mu^2/4) n^2$, and each edge is coloured \(c\) independently with probability \(1/n\), 
			\Cref{lemma:linear-matching} yields that,
			with probability at least \(1-\exp(-\Omega(\lambda' n))\)
			, there exists a matching \(M \subseteq D_1[Y,U]\) of size at least \(\mclb n\) with all edges coloured \(c\).

			For \((c_1, c_2, c_3) \in \cols_0\), say an edge \(xz\) in \(E_{D}(X,Z)\) is \emph{good} for \((c_1, c_2, c_3)\in \cols_0\), if there exist $u \in U$ and $y \in Y$ such that $yz \in E(H)$, $xy, yu, zu \in E(D)$, 
			$\cols(xy)=c_3$, $\cols(yu)=c$, and $\cols(zu)=c_1$.
			Notice that this definition does not depend on the colours of  $E_{D}(X,Z)$.
			Observe also that if $xz$ is good for $(c_1,c_2,c_3) \in \cols_0$ and \(\cols(xz) = c_2\), then $(x,y,z,u)$ satisfies the requirements of the lemma.
			Therefore, to complete the proof it suffices to show that, with high probability, for some \((c_1, c_2, c_3)\in \cols_0\) there exists at least one good edge  \(xz\in E_{D}(X,Z)\) with \(\cols(xz) = c_2\).
			For that, the crucial argument is the following claim, where we show that, with very high probability, there are many good edges for $(c_1,c_2,c_3)$.
			We do so by `extending' \(M\) with a monochromatic matching \(M_1 \subseteq D_1[Z,U]\) in colour $c_1$, and a monochromatic matching \(M_3 \subseteq D_1[X,Y]\) in colour \(c_3\).
			
			For $(c_1,c_2,c_3) \in \cols_0$, write
			\[
				G(c_1,c_2,c_3) := \{ e\in E_{D}(X,Z): e \text{ is good for } (c_1,c_2,c_3) \}.
			\]

			\begin{claim} \label{claim:good-edges}
				Fix \((c_1,c_2,c_3) \in \cols_0\).
				With probability at least
				\(1-\exp(- \Omega(\lambda' n))\), we have
				\(
				\abs{G(c_1,c_2,c_3)} \ge \goodmiddledges n.
				\)
			\end{claim}
			\begin{proof}
				For $y \in Y \cap V(M)$, denote by $M(y)$ the neighbour of $y$ in $M$.
				Let $D_2$ be the subdigraph of $D_1$ obtained by removing all edges in $Z \times U$ except edges of the form $zM(y)$ satisfying $yz \in E(H)$ and $|\Nm(y) \cap \Nm(z) \cap X| \ge (\rho^2 \mu / 2)n$, for some $y \in Y$.

				We claim that $e(D_2[Z, U]) \ge (\alpha \rho \mu \gamma / 8)n^2$.
				Indeed, given $y \in Y \cap V(M)$, by $(\eps,\rho)$-super-regularity, we have $|\Nm_{D_1}(y) \cap X| \ge \rho|X|$, and all but at most $\eps |Z|$ vertices $z \in Z$ satisfy 
                $$|\Nm_{D_1}(z) \cap \Nm_{D_1}(y) \cap  X| \ge (\rho - \eps)|\Nm_{D_1}(y) \cap X| \ge (\rho^2/2) |X| = (\rho^2\mu/2)n.$$
				Thus, using that $|\Nm_{D_1}(M(y)) \cap \Np_H(y) \cap Z| \ge (\alpha \rho \mu / 4)n = (\alpha \rho / 4) |Z|$, by choice of $D_1$, 
				\begin{align*}
					\big|\Nm_{D_2}(M(y)) \cap Z\big| 
					& \ge \big|\Nm_{D_1}(M(y)) \cap \Np_H(y) \cap Z\big| - \eps |Z| \\
					& \ge (\alpha \rho / 4 - \eps)|Z| \ge (\alpha \rho / 8)|Z| = (\alpha \rho \mu / 8) n.
				\end{align*}
				Hence, $e(D_2[Z, U]) \ge |M| \cdot (\alpha \rho \mu / 8) n \ge (\alpha \rho \mu \gamma / 8)n^2$, as claimed.
				Reveal the colours of $E_{D_2}[Z,U]$.
				Since each edge is coloured \(c_1\) independently with probability \(1/n\), using $e(D_2[Z, U]) \ge (\alpha \rho \mu \gamma / 8)n^2$,  by \Cref{lemma:linear-matching},
				with probability at least
				\(1-\exp(-\Omega(\lambda' n))\),
				there is a matching \(M_1\) in \(D_2[Z, U]\) coloured \(c_1\) that has size at least \(\mclb_1 n\).

				Denote by $M_1(y)$ the neighbour of $M(y)$ in $M_1$ (if it exists).
				Notice that, by the definition of $D_2$, the matching $M_1$ `extends' the matching $M$ in the sense that \(V(M_1)\cap U \subseteq V(M)\), and, moreover, for every $y \in Y$ such that $M_1(y)$ exists, so does \(M(y)\), and $y M_1(y) \in E(H)$.
				
				Next, we will find a large matching \(M_3\) coloured \(c_3\) which, along with \(M_1\) and \(M\), will give us a large number of good edges for \((c_1,c_2,c_3)\).
				To this end, let \(D_3\) be the spanning subgraph of \(D_2\) obtained by 
                removing all edges \(xy\in X\times Y\) except those for which \(M_1(y)\) exists and \(xM_1(y)\)  is an edge in \(D_2\).

				We claim that $e(D_3[X,Y]) \ge (\rho^2 \mu \gamma_1 / 2)n^2$.
				Indeed, let $y \in Y$ be such that $M(y)$ and $M_1(y)$ exist. Then $|\Nm_{D_2}(y) \cap \Nm_{D_2}(M_1(y)) \cap X| \ge (\rho^2 \mu / 2)n$, by choice of $D_2$, showing $|\Nm_{D_3}(y)| \ge (\rho^2 \mu / 2)n$. Hence, $e(D_3[X, Y]) \ge |M_1| \cdot (\rho^2 \mu / 2)n \ge (\rho^2 \mu \gamma_1 / 2)n^2$.

				Reveal the random colouring of $D_3[X,Y]$.
				Since each edge of \(D_3[X,Y]\) is coloured \(c_3\) independently with probability \(1/n\),
				using \Cref{lemma:linear-matching}, with probability 
				at least
				\(1-\exp(-\Omega(\lambda' n))\),
				there exists a matching \(M_3\) in \(D_3[X,Y]\) coloured \(c_3\) that has size
				at least \(\mclb_3 n\).

				For $y \in Y$, denote by $M_3(y)$ the neighbour of $y$ in $M_3$ (if it exists).
				Notice that, by the definition of $D_3$, the matching $M_3$ `extends' the matching $M_1$ in the sense that if $M_3(y)$ exists then so does $M_1(y)$ and \(M_3(y) M_1(y) \in E(D_3)\).
				
				Let
				\[
					G_0(c_1,c_2,c_3):= \{ M_3(y) M_1(y) : y \in Y \text{ and $M_1(y), M_3(y)$ exist}\}.
				\]
				Notice that $|G_0(c_1,c_2,c_3)| = |M_3| \ge \gamma_3 n$. 
				Moreover, $M_3(y) M_1(y)$ is good for $(c_1,c_2,c_3)$, as can be seen by taking $u = M(y)$, since we have: $yM_1(y) \in E(H)$; $M_3(y)y, yM(y), M_1(y)M(y) \in E(D)$; $\cols(M_3(y)y) = c_3$, $\cols(yM(y)) = c$, $\cols(M_1(y)M(y)) = c_1$. Hence $|G(c_1,c_2,c_3)| \ge |G_0(c_1,c_2,c_3)| \ge \gamma_3n$, with probability at least $1 - \exp(-\Omega(\lambda n))$, as required for the claim.
			\end{proof}
			
			By the union bound over the $q_0 n$ triples \((c_1, c_2, c_3) \in \cols_0\),
			\Cref{claim:good-edges} implies that,
			with probability 
			at least
			\(1-\exp(-\Omega(\lambda' n))\),
			for each \((c_1, c_2, c_3) \in \cols_0\), we have
			\(\abs{G(c_1,c_2,c_3)} \ge \goodmiddledges n\).
			We finally show that, with probability at least $1 - \exp(-\Omega(\lambda' n))$, for some $(c_1,c_2,c_3) \in \cols_0$ there exists $e \in G(c_1,c_2,c_3)$ satisfying $\cols(e)=c_2$.
			
			For $e \in E_{D_3}(X,Z)$, let
			\[
				G'(e) := \{c_2 \in \cols: 
				\text{there exist \(c_1, c_3\) such that } 
				(c_1,c_2,c_3) \in \cols_0 \text{ and } e \in G(c_1, c_2, c_3)\}.
			\]
			Then, using that no two triples in \(\cols_0\) share a colour, we have
			\[
				\sum_{e \in E_{D}(X,Z)} \abs{G'(e)} = \sum_{(c_1,c_2,c_3) \in \cols_0} \abs{G(c_1,c_2,c_3)}
				\ge 
				\abs{\cols_0} \cdot\goodmiddledges n
				= q_0 \goodmiddledges  n^2.
			\]
			Now we reveal the colours of \(E_{D_3}(X,Z)\).
			For \(e\in E_{D_3}(X,Z)\), let \(A_e\) be the event that \(e\) gets a good colour, i.e.\ 
			\(\cols(e) \in G'(e)\).
			Each edge is coloured independently,
			so the events \(A_e\) are mutually independent, and it holds that
			 \(\prob{A_e} = \abs{ G'(e)} /n \).
			Hence, the probability that no event $A_e$ holds is
			\begin{align*}
				\prod_{e \in E_{D_3}(X,Z)} \left(1- \prob{A_e}\right)
				& \le \exp \left(
					-\sum_{e \in E_{D_3}(X,Z)} \prob{A_e}
					\right) \\
				& = 
				\exp \left(-\sum_{e \in E_{D_3}(X,Z)} \frac{|G'(e)|}{n}\right)
				\le 
				\exp \left( -q_0 \goodmiddledges n \right).
			\end{align*}
			Hence, with probability at least 
			\(1-\exp( -q_0 \goodmiddledges n)\), 
			at least one edge gets a good colour, i.e.\ there exists $e\in E_{D_3}(X,Z)$ such that $\cols(e) \in G'(e)$, as required for the lemma.
			Altogether, all the required events hold simultaneously with probability at least $1 - \exp(-\Omega(\lambda' n)) - \exp(-q_0 \goodmiddledges n) \ge 1 - \exp(-\lambda n)$, as required.
		\end{proof}

		Next, we show how to find a copy of $K_{2,4}$ with the orientation in \Cref{fig:absorber}.
		As above, we additionally make sure that the edges $zu,xz,xy$ are coloured according to a triple of colours from a linear set $\cols_0$ of pairwise disjoint triples of distinct colours.
        When applying \Cref{lemma:K24} to find a \((v,c)\)-absorber (cf.~\Cref{lemma:absorbers-exist} and \Cref{section:proof-main-lemma}),  \(\cols_0\) will be precisely the set of colour triples seen on a collection of rainbow triangles intersecting only on \(v\).

		\begin{lemma}[Finding a \(K_{2,4}\)] \label{lemma:K24}
			Let \(1/n \ll \lambda \ll \delta, q_0\).
			Let \(\cols=[n]\), $c \in \cols$, and 
			\(\cols_0 \subseteq \cols^3\) be a collection of colour triples that are pairwise disjoint and avoid $c$, with \(\abs{\cols_0} = q_0 n\).
			Let \(D\) be a digraph on $n$ vertices with minimum semidegree at least \(\delta n\), which is coloured uniformly in \(\cols\).

			Then, with probability at least 
			\(1-\exp(-\lambda n)\), 
			the following holds.
			There exists an oriented copy \(K\) of \(K_{2,4}\), with \(V(K) = \{x,y,z,u,w_1,w_2\} \)  and \(E(K) = \{xz,xy,yu,zu, zw_1, \allowbreak yw_1, w_2z, w_2y\} \).
			Moreover, the edges of \(K\) are coloured as follows:
			\((\cols(zu), \cols(xz), \cols(xy)) \in \cols_0\), \(c=\cols(yu)\), and \(\cols(y w_1) = \cols(zw_1) \neq \cols(w_2 y) = \cols(w_2 z)\), with \(\cols(y w_1), \cols(w_2 y)\) disjoint from \(\{c, \cols(zu), \cols(xz), \cols(xy)\}\).
		\end{lemma}
		\begin{proof}
			
			Let $M', M, \alpha', \alpha'', \lambda'$ satisfy
			\[
				\lambda \ll \lambda' \ll 1/M \ll 1/M' \ll \eps \ll \alpha'' \ll \alpha' \ll \rho \ll \delta, q_0.
			\]
			Consider a partition \( (V_0, \hdots, V_k)\), with \(M' \le k\le M\), of \(V(D)\) and a spanning subgraph \(D'\), both given by \Cref{lemma:regularity} with parameters \(\eps_{\ref{lemma:regularity}}=\eps, d_{\ref{lemma:regularity}}=\rho\) and \(M'\).
			Let \(R\) be the reduced digraph associated with this partition, that is the digraph on \([k]\) where \(ij \in E(R)\) if and only if \(e_{D'}(V_i, V_j) >0\).
			Observe that if $ij \in E(R)$ then $D'[V_i,V_j]$ is $\eps$-regular and has density at least $\rho$.
			
			\begin{claim}
			\label{claim:copy_K_{2,4}}
				The digraph \(R\) contains a copy of \(K\).
			\end{claim}
			\begin{proof}
				We observe that the minimum semidegree of $R$ satisfies $\delta^+(R) \ge \delta k/2$ because, otherwise, there would be vertices with out-degree at most $(\delta k / 2) \cdot (n/k) < \delta n - (\rho+\eps)n$ in $D'$, contradicting that \(\deg_{D'}^+(x) > \deg_D^+(x) - (\rho+\eps)n\) for all  \(x\in V(D)\).

				We will use, twice, the following assertion.
				\begin{enumerate}[label = (\ding{96})]
					\item \label{itm:fact}
						If $U$ is a set of at least $4/\delta$ vertices in $R$, then there are two distinct vertices $u_1, u_2 \in U$ such that $|\Np(u_1) \cap \Np(u_2)| \ge \frac{\delta^2 k}{10}$.
				\end{enumerate}
				First, let us prove \ref{itm:fact}. Suppose that $U$ is a set contradicting \ref{itm:fact}. Write $\ell = |U|$, and note that we may assume $\ell = \ceil{4/\delta}$, so $4/\delta \le \ell \le 5/\delta$. 
				Then,
				\begin{align*}
					k \ge \left|\bigcup_{u \in U}\Np(u)\right|
					& \ge \sum_{u \in U}|\Np(u)| - \sum_{u, v \in U \text{ distinct}} |\Np(u) \cap \Np(v)| \\
					& \ge \ell \cdot \frac{\delta k}{2} - \binom{\ell}{2} \frac{\delta^2 k}{10} \\
					& = \frac{\ell \delta k}{2}\left(1 - \frac{(\ell-1)\delta}{10}\right)
					> k,
				\end{align*}
				a contradiction.

				Now, using \ref{itm:fact} (with $U = V(R)$), we can find two distinct vertices $x,w_2$ with $|\Np(x) \cap \Np(w_2)| \ge \delta^2k/10 \ge 4/\delta$. Apply \ref{itm:fact} again, with $U = \Np(x) \cap \Np(w_2)$ to find two distinct out-neighbours of both $x$ and $w_2$, denoted $y,z$, such that $|\Np(y) \cap \Np(z)| \ge \delta^2k/10 \ge 4$. Now pick distinct $u,w_1 \in (\Np(y) \cap \Np(z)) \setminus \{x,w_2\}$. Then $\{x,y,z,u,w_1,w_2\}$ spans a copy of $K$, as required.
			\end{proof}
			Let \(X,Y,U,Z,W_1,W_2\subseteq V(D)\) be the vertex clusters corresponding to the vertices of the copy of \(K\) in \(R\) given by \Cref{claim:copy_K_{2,4}}.
			Using \Cref{lemma:reg:few-bad}, by removing up to \(O(\eps n)\) vertices from each cluster, we can ensure that $(X,Y)$, $(X,Z)$, $(W_1, Y)$, $(W_1, Z)$, $(Y, U)$, $(Z,U)$, $(Y, W_1)$, $(Z,W_1)$ are \((\eps, \rho/2)\)-super-regular and that $X,Y,Z,U,W_1,W_2$ all have the same size, which we denote by $n'$.
			Thus, without loss of generality, we assume this is the case for the remainder of the proof.
			
			Delete some triples from \(\cols_0\) if necessary so that $\cols_0$ covers at most $n/2-1$ colours, and let $\cols_1$ and $\cols_2$ be disjoint sets of $n/4$ colours that do not appear in triples in $\cols_0$ and are not $c$.
			For each edge of \(E(Y \cup Z, W_1)\), and for these edges only,
			reveal if its colour belongs to \(\cols_1\), and delete if it does not.
			Similarly, for each edge of $E(W_2, Y \cup Z)$, and for these edges only, reveal if its colour belongs to $\cols_2$, and delete it if not.
			This amounts to sampling each edge in $E(Y \cup Z, W_1) \cup E(W_2, Y \cup Z)$ with probability \(1/4\).
			Hence, by \Cref{lemma:reg:edge-sampling}, with probability at least $1 - \exp(-\Omega(\rho n))$,
			the remaining edges in $E(Y \cup Z, W_1) \cup E(W_2, Y \cup Z)$
			span \((\eps,\rho/16)\)-super-regular pairs.
			
			Now reveal the colours of the remaining edges in $E(Y \cup Z, W_1)$, which are chosen uniformly from $\cols_1$. By \Cref{lemma:ABC}, with $(n,\eps,\alpha,\rho,q,A,B,C)_{\ref{lemma:ABC}} = (n',\eps,1,\rho/16,n/(4n'), Y, W_1, Z)$ and $H$ with $E(H) = Y \times Z$, with probability at least $1 - \exp(-\lambda' n)$, the following holds: there is a subgraph $H' \subseteq H$ with $e(H') \ge \alpha' (n')^2$, such that for every edge $yz \in E(H')$, there is a vertex $w_1 \in W_1$ such that $\cols(yw_1) = \cols(zw_1)$. 
			Next, reveal the colours of the remaining edges in $E(W_2, Y \cup Z)$; they are coloured uniformly in $\cols_1$. By the same lemma, with $(n,\eps,\alpha,\rho,q,A,B,C,H)_{\ref{lemma:ABC}} = (n',\eps,\alpha',\rho/16,n/(4n'),Y,W_2,Z,H')$, with probability at least $1 - \exp(-\lambda' n)$, the following holds: there is a subgraph $H'' \subseteq H'$ with $e(H'') \ge \alpha''(n')^2$, such that for every edge $yz \in E(H'')$ there exists $w_2 \in W_2$ such that $\cols(w_2y) = \cols(w_2z)$.\footnote{Technically, here we apply the lemma to the digraph obtained by reversing the direction of edges in $E(W_2, Y \cup Z)$.}

			Finally, by \Cref{lemma:XYZU}, applied with $(\eps,\alpha,\rho,\mu,H)_{\ref{lemma:XYZU}} = (\eps,\alpha'',\rho/16,n'/n,H'')$, with probability at least $1 - \exp(-\lambda' n)$, there is a quadruple $(x,y,z,u) \in X \times Y \times Z \times U$ such that: $xy,xz,yu,zu$ are edges in $D$; $yz$ is an edge in $H''$; $(\cols(zu),\cols(xz),\cols(xy)) \in \cols_0$; and $\cols(yu) = c$. By choice of $H''$, there are vertices $w_1 \in W_1$ and $w_2 \in W_2$ such that $yw_1,zw_1,w_2u,w_2z$ are edges and $\cols(zw_1) = \cols(yw_1) \in \cols_1$ and $\cols(w_2z) = \cols(w_2y) \in \cols_2$.
			Altogether, all the required events hold with probability at least $1 - \exp(-\Omega(\rho n)) - \exp(-\Omega(\lambda' n)) \ge 1 - \exp(-\lambda n)$.
		\end{proof}

	\subsection{Proof of Lemma~\ref{lemma:absorbers-exist}}

	\begin{proof}[Proof of \Cref{lemma:absorbers-exist}]    
		Let $\lambda, \rho, \rho'$ satisfy \(\nu \ll \prblem, \pathb, \pathb' \ll \delta\).
		Fix \(v\in V(D),\, c\in \cols\) and 
		\(V' \subseteq V(D),\, \cols' \subseteq \cols\) of size at least $(1 - \nu)n$.

		For the next claim, it is useful to refer to \Cref{fig:absorber}.
		\begin{claim} \label{claim:sq-tr}
			With probability \(1-\exp(-\Omega(\prbsqtr n))\)
			the following holds.
			In
			\(D[V']\)
			there exist an oriented triangle \(T\) with \(V(T) = \{v,v_1,v_2\}\) and \(E(T) = \{vv_2, v_1 v, v_1 v_2\}\), and
			an oriented copy \(K\) of \(K_{2,4}\), with \(V(K) = \{x,y,z,u, w_1, w_2\} \)  and 
			\(E(K) = \{xz, xy, yu, zu, yw_1, zw_1, \allowbreak  w_2y, w_2z\} \).
			Moreover, the edges of $K$ and $T$ are coloured as follows.
			\(\cols(yu) = c\),
			\(\cols(xz) = \cols(v_1 v_2)\),
			\(\cols(xy) = \cols(v_1v)\),
			\(\cols(zu) = \cols(vv_2)\),
			$\cols(zw_1)=\cols(yw_1)$ and $\cols(w_2 z)=\cols(w_2y)$.
		\end{claim}
		\begin{proof}
			Let \((\Vtr, \Vsq)\) be a random partition of \(V'\).
			Then, from Chernoff's bound, a union bound over the vertices of \(V'\), and $\nu \ll 1$, with probability \(1-\exp(-\Omega(\delta n))
			\), the digraphs \(\ddn{\delta}{n}[\Vtr]\), \(\ddn{\delta}{n}[\Vsq]\) have minimum semidegree at least \(\delta n /3\), and \(\abs{\Vtr}, \abs{\Vsq} \ge n/3\).
		
			First reveal the random edges and colours of $D[\Vtr]$.
			
			Then, by \Cref{lemma:con}, 
			with probability \(1-\exp(-\Omega(\prbsqtr n))\),
			there is a collection 
			\(\Delta_v\) of \(\rho n\) rainbow triangles $\{v,v_1,v_2\}$, oriented as in the claim, that pairwise intersect only on \(v\), are pairwise colour-disjoint, and with \(V(\Delta_v) \subseteq \Vtr \cup \{v \} \).
			For $T \in \Delta_v$, denote its vertices by $v$, $v_1(T)$, $v_2(T)$ such that \(E(T) = \{vv_2(T), v_1(T) v, v_1(T) v_2(T)\}\).
			Let 
			$\cols_v$ be the collection of triples $\left(\cols(vv_2(T)), \cols(v_1(T)v_2(T)), \cols(v_1(T)v)\right)$ with $T \in \Delta_v$, whose three colours are in $\cols'$. 
			Then $\abs{\cols_v} \ge \rho n - 3\nu n \ge (\rho/2)n$.

			Next, reveal the colours of edges in $D[\Vsq]$.
			By setting \(\cols_0 = \cols_v\) in \Cref{lemma:K24}, it follows that, with probability \(1-\exp(-\Omega(\prbsqtr n))\),
			\(D[\Vsq]\) contains a subgraph \(K\) isomorphic to \(K_{2,4}\) with the orientation as in the claim, whose colours match those of a 
			triangle \(T \in \Delta_v\) as in the statement of the claim.

			We fail to find a \(T\) or \(K\) as required with probability at most \(\exp(-\Omega(\prbsqtr n))\).
		\end{proof}

		By \Cref{lemma:con}, with probability \(1-\exp(-\Omega(\prcon n))\), for any two vertices \(a,b\in V'\) 
		there are \(\pathb' n\) rainbow directed paths of length three from \(a\) to \(b\) which are pairwise colour disjoint and internally vertex disjoint.
		Hence, with probability
		\(1-\exp(-\Omega(\prcon n))\), this and the conclusion of \Cref{claim:sq-tr} hold simultaneously.

		Then, using $\nu \ll \pathb'$, there exist two colour- and vertex-disjoint rainbow directed paths $P_1, P_2$ of length $3$ such that: $P_1$ is directed from $v_2$ to $x$; $P_2$ is directed from \(w_1\) to \(w_2\); the interiors of $P_1$ and $P_2$ are in $V' \setminus (V(K) \cup V(T))$; and the colours of $P_1$ and $P_2$ are in $\cols' \setminus (\cols(K) \cup \cols(T))$.
		Then the graph $A_{v,c}$, defined as
		\[
			A_{v,c} = K\, \cup\, T \,\cup\, P_1\, \cup\, P_2,
		\]
		is a \((v,c)\)-absorber: 
		the 
		\((v,c)\)-absorbing path  
		and the \((v,c)\)-avoiding path are 
		\begin{equation*}
			v_1 v v_2 P_1 x z w_1 P_2 w_2 y u \quad \text{and} \quad
			v_1 v_2 P_1 x y w_1 P_2 w_2 z u 
		\end{equation*}
		 and it is straightforward to check they satisfy~\Cref{def:absorber} and have 13 vertices.

		The number of \(V' \subseteq V\) of size at least \((1-\gadgetb) n\) is at most
		\(n \binom{n}{\gadgetb n} = \exp(O\left(\gadgetb \log \gadgetb\right) n)\),
		and the same bound holds for the number of 
		\(\cols'\subseteq \cols\) of the same size.
		Using 
		\(\gadgetb \ll \prcon\),
		the probability we fail to find an absorber for some \(v,\, c,\, V',\, \cols'\), is by the union bound, at most
		\[
			n^2\e^{O\left(\gadgetb\log\gadgetb \right)n} \cdot \e^{-\Omega(\prcon n)}
			 = o(1).\qedhere
		\]
	\end{proof}

\section{Proof of Lemma~\ref{lemma:main}} \label{section:proof-main-lemma}

	In this section we prove \Cref{lemma:main}. The following lemma enables us to cover any small set of vertices and colours.
	\begin{lemma}[Flexible sets]\label{lemma:cover}
		Let \(1/n \ll 1/C \ll \coverub\ll \flxcnst \ll \delta \le 1\).
		Let $D_0$ be a digraph on \(n\) vertices with minimum semidegree at least \(\delta n\),
		and suppose that \(D \sim D_0\cup \dnp{n}{C/n}\) is uniformly coloured in \(\cols = [n]\).
		Then there exist \(\Vflx\subseteq V\),
		\(\colsflx \subseteq \cols\)
		of size \(2 \flxcnst n\) 
		such that with high probability 
		the following holds.
		For all distinct \(u,v \in V\), \(c\in \cols\),
		and \(\Vflx' \subseteq \Vflx\), \(\colsflx' \subseteq \colsflx\) of size at least \((2\flxcnst - \coverub)n\),
		there exists a rainbow directed path of length seven from $u$ to $v$, with internal vertices in \(\Vflx'\) and colours in \(\colsflx' \cup \{c\}\), that contains the colour \(c\).
	\end{lemma}

	\begin{proof}
		Let $\gamma$ satisfy $\coverub \ll \gamma \ll \flxcnst$.
		
		For a colour $c$, let $M_c$ be a largest oriented matching of colour $c$ in $D$, and for distinct vertices $u,v$, let $\calP_{u,v}$ be a largest collection of pairwise vertex- and colour-disjoint rainbow directed paths of length three with endpoints $u,v$.
		By \Cref{lemma:linear-matching,lemma:con}, with probability at least $1 - \exp(-\gamma n)$, we have $|M_c| \ge \gamma n$ and $|\calP_{u,v}| \ge \gamma n$ for every colour $c$ and distinct vertices $u,v$.
		
		\def \Vrand {V'}
		\def \Crand {\cols'}
		Let $\Vrand$ be a random subset of $V$, obtained by including each vertex independently with probability $\flxcnst$, and let $\Crand$ be a random subset of $\cols$, obtained by including each colour independently with probability $\mu$.
		
		Then, by Chernoff and union bounds, with high probability,
		the following properties hold.
		\begin{itemize}
			\item 
				$|\Vrand|, |\Crand| \le 2\mu n$,
			\item
				at least $\mu^2 \gamma n/2$ edges in $M_c$ have both endpoints in $\Vrand$, for every $c \in \cols$,
			\item
				at least $\mu^5 \gamma n/2$ paths in $\calP_{u,v}$ have their interior vertices in $\Vrand$ and all colours in $\Crand$, for all distinct $u,v \in V$.
		\end{itemize}
		
		Suppose that all three properties hold, and let $\Vflex$ be a subset of $V$ that contains $\Vrand$ and has size $2\flxcnst n$ and let $\Cflex$ be a subset of $\cols$ that contains $\Crand$ and has size $2\flxcnst n$. 
		
		We show that these sets satisfy the requirements of the lemma. Indeed, fix $u,v, c$ and $\Vflex',\Cflex'$ as in the lemma.
		Then, as $\coverub \ll \mu, \gamma$, there is an edge $e = xy \in M_c$ with both ends in $\Vflex'$.
		Similarly, there are paths $P_1 \in \calP_{u,x}$ and $P_2 \in \calP_{y,v}$ that are vertex- and colour-disjoint, their interiors are in $\Vflex'$, and their colours are in $\Cflex' \setminus \{c\}$. Then $P_1 \cup e \cup P_2$ is a path that satisfies the requirements of the lemma.
	\end{proof}

	We will put together several \( (v,c) \)-absorbers to construct the digraph $\Habs$ in Lemma~\ref{lemma:main}, by having a \((v,c)\)-absorber for each edge of a bipartite graph which has the following property.
	This follows an idea introduced by Montgomery~\cite{montgomery}, which was adapted to the rainbow setting by Gould, Kelly, K\"uhn and Osthus~\cite{kuhn-osthus}.

	\begin{definition}[Definition 3.3,~\cite{kuhn-osthus}] \label{def:robustly-matchable}
		Let $H$ be a balanced bipartite graph with bipartition $(A,B)$.
		We say $H$ is \emph{robustly matchable} with respect to $A\prm,\, B\prm$, for some \(A\prm \subseteq A\) and
		\(B\prm \subseteq B\) of equal size,
		if for every pair of sets $X\subseteq A\prm, Y\subseteq B\prm$ with $\abs{X} = \abs{Y} \leq  \abs{A\prm}/2,$
		there is a perfect matching in $H[A\setminus X, B\setminus Y]$.
		We call \(A', B'\) the \emph{flexible sets} of \(H\).
	\end{definition}

	\begin{proposition}[Lemma 4.5, \cite{kuhn-osthus}] \label{prop:rmbg}
	For every large enough $m\in \nats$, there exists a $256$-regular bipartite graph with bipartition $(A,B)$ and $\abs{A} = \abs{B} = 7m$,
	which is robustly matchable with respect to some $A\prm \subseteq A, B\prm \subseteq B$ with $\abs{A\prm} = \abs{B\prm} = 2m$.
	\end{proposition}

	\begin{proof}[Proof of \Cref{lemma:main}]
		Let $\coverub, \flxcnst, \gadgetb$ satisfy
        \(\eta \ll \coverub\ll \flxcnst\ll \gadgetb \ll \gamma\).

		By the union bound, the conclusions of \Cref{lemma:con,lemma:cover,lemma:absorbers-exist} hold simultaneously with high probability.
		Let \(\Vflx,\) \(\colsflx\) be sets as given by \Cref{lemma:cover}; so $|\Vflx| = |\colsflx| = 2\flxcnst n$.
		
		Let  \(\Vbuf,\, \colsbuf\) be arbitrary subsets of \(V\setminus \Vflx,\, \cols\setminus \colsflx\) of size \(5 \flxcnst n\).
		Let \(H\) be a bipartite graph on 
		\((\Vflx \cup \Vbuf,\, \colsflx \cup \colsbuf)\)
		that is isomorphic to the graph in \Cref{prop:rmbg} such that \(\Vflx,\, \colsflx\) are the flexible sets.

		\begin{claim} \label{claim:absorbing-template}
			There is collection of absorbers $A_{v,c}$ on $13$ vertices each and rainbow directed paths $P_{v,c}$ of length three each, for every edge $vc$ in $H$, with the following properties:
			the internal vertices of $A_{v,c}$ and of $P_{v,c}$ are pairwise disjoint and disjoint of $\Vflx \cup \Vbuf$; 
			the internal colours of $A_{v,c}$ and the colours of $P_{v,c}$ are pairwise disjoint and disjoint of $\colsflx \cup \colsbuf$; and for some ordering of the edges of $H$, the path $P_{v,c}$ starts with the last vertex of $A_{v',c'}$ and ends with the first vertex of $A_{v,c}$, where $v'c'$ is the predecessor of $vc$ in the ordering (with $P_{vc} = \emptyset$ the trivial path for the first edge $vc$).
		\end{claim}
		
		\begin{proof}
			Let \(H_0\) be a maximal subgraph of 
			\(H\) with some ordering of its edges, for which we can find a collection of absorbers and paths as in the claim.
			Suppose for contradiction \(H_0 \neq H\) and let
			\(v_1 c_1 \in E(H\setminus H_0)\) and \(v_0 c_0\) be the last edge of \(H_0\) in the ordering that has absorber \(A_{v_0, c_0}\).

			Let \(V_0, \cols_0\) be the union of the vertices and colours spanned by the absorbers for \(E(H_0)\), the paths connecting them, and
			\(\Vflx \cup \Vbuf\), 
			\(\colsflx \cup \colsbuf\).
			Then, since each absorber has $13$ vertices 
			and each path connecting consecutive absorbers has four vertices, we have
			\[
			\abs{V_0}, \abs{\cols_0} = O(e(H_0)) = O(\flxcnst n) < \gadgetb n/2.
			\]
			Hence, by the assumption that \Cref{lemma:absorbers-exist} holds, there exists a \((v_1, c_1)\)-absorber
			\(A_{v_1,c_1}\) on 13 vertices
			with internal vertices and internal colours disjoint from 
			\(V_0\) and \(\cols_0\).
			Moreover, 
			from \Cref{lemma:con} there exists a rainbow directed path $P_{v_1, c_1}$ of length three
			between the last vertex of \(A_{v_0, c_0}\) and the first vertex of \(A_{v_1,c_1}\), with internal vertices disjoint of 
			\(V_0 \cup V(A_{v_1, c_1})\) and colours disjoint of \(\cols_0 \cup \cols(A_{v_1, c_1})\).
			Then the subgraph of \(H\) with edges 
			\(E(H_0) \cup \{ v_1 c_1 \}\) satisfies the conditions of the claim and properly contains 
			\(H_0\), contradicting the maximality of \(H_0\).
		\end{proof}

		Set
		\[
			\Habs:= \bigcup_{vc \in E(H)} (A_{v,c} \cup P_{v,c})\, ,
		\]
		where $A_{v,c}$ and $P_{v,c}$ are as in \Cref{claim:absorbing-template}, with $P_{v,c}=\emptyset$ for the first \(vc\),
		
		Let $w$ be the first vertex in the first absorber in $\Habs$ with the ordering given by \Cref{claim:absorbing-template}, and let $w'$ be the last vertex of the last absorber.
		We will now show how to construct a directed path \(Q\) from $x$ to $y$ as in the statement of the lemma, given $V' \subseteq V \setminus V(\Habs)$ and $\cols' \subseteq\cols \setminus \cols(\Habs)$ of size between $2$ and $\eta n$.
		Let \(c_0 \in \cols'\).
		From \Cref{lemma:cover}, there exists a rainbow directed path \(Q_1\) from $x$ to $w$, with internal vertices in \(\Vflx\) and colours in \(\colsflx \cup \{c_0\}\), which includes the colour $c_0$ and has length 7.

		\begin{claim}
			There exists a rainbow directed path \(Q_2\) from $w'$ to $y$,  with internal vertices 
			\(\Vflx\dprm \cup (V'\setminus \{x\})\),
			and colours \(\colsflx\dprm \cup (\cols'\setminus \{c_0\})\),
			for some \(\Vflx\dprm \subseteq \Vflx \setminus V(Q_1)\),
			\(\colsflx\dprm \subseteq \colsflx \setminus \cols(Q_1)\) 
			with
			\(\abs{\Vflx\dprm} = \abs{\colsflx\dprm} \le \flxcnst n - 7\).
		\end{claim}
		\begin{proof}
			The Claim will follow by applying greedily \Cref{lemma:cover}. 

			Let $v_0, \ldots, v_k$ be an enumeration of the vertices in $(V' \cup \{w'\}) \setminus \{x\}$, with \(v_0=w'\) and $v_k = y$, and let $c_1, \ldots, c_k$ be an enumeration of $\cols' \setminus \{c_0\}$.
			We claim that there exist directed paths $P_1, \ldots, P_k$ as follows: $P_i$ is a rainbow directed path from $v_{i-1}$ to $v_i$ of length 7; the interiors of the paths $P_i$ are pairwise vertex-disjoint, contained in \(\Vflx\), and disjoint of $V' \cup \{w'\}$; $P_i$ contains an edge coloured $c_i$, and all other colours are in \(\Cflex\); and no colour appears on more than one path $P_i$.

			To see this, suppose that $P_1, \ldots, P_{i-1}$ are defined. 
			Let \(\Vflx' = \Vflx \setminus (V(P_1) \cup \ldots \cup V(P_{i-1}) \cup V(Q_1))\),
			\(\colsflx' = \colsflx \setminus (\cols(P_1) \cup \ldots \cup \cols(P_{i-1}) \cup \cols(Q_1))\).
			Then $|\Vflx'| \ge 2\flxcnst n - 7k \ge (2\mu - \coverub)n$, using that $k \le \eta n$, $|V(Q_1)|, |V(P_j)| = 8$, and $\eta \ll \coverub$, and similarly $|\colsflx'| \ge (2\mu - \coverub)n$.
			Hence, by \Cref{lemma:cover} there is a rainbow directed path \(P_i\) from 
			\(v_{i-1}\) to \(v_i\) of length 7, that contains an edge with colour \(c_i\), 
			and whose internal vertices and other colours are in \(\Vflx'\) and \(\colsflx'\), as required.

			Write $Q_2 = P_1 \ldots P_k$.
			Let \(\Vflx\dprm = V(Q_2) \cap \Vflx\),
			\(\colsflx \dprm = \cols(Q_2) \cap \colsflx\).
			Then
			\(
			\abs{\Vflx\dprm} = \abs{\colsflx \dprm} < \abs{Q_2} \le \coverub n < \flxcnst n -7,
			\)
			as required.
		\end{proof}

		Let $\Vflx''' = (V(Q_1) \cup V(Q_2)) \cap \Vflx$ and $\colsflx''' = (\cols(Q_1) \cup \cols(Q_2)) \cap \colsflx$. 
		Then
		we have 
		$\abs{\Vflx'''} = \abs{\colsflx'''} \le \flxcnst n$.
		Hence by choice of $H$ there is a matching $M'$ between $\Vflx \setminus \Vflx'''$ and $\colsflx \setminus \colsflx'''$. 
		
		For \(vc\in E(H)\)
		let \(Q_{M'}(vc)\) be the \((v,c)\)-absorbing path of \(A_{v,c}\) if \(vc\in E(M')\) and the avoiding path otherwise.
		Let

		\[
			\Qabs = \bigcup_{vc\in E(H)} (Q_{M'}(vc)\ \cup P_{v,c}).
		\]
		Then \(\Qabs\) is a rainbow directed path that is spanning in 
		\(V(\Habs)\setminus \Vflx'''\)
		and 
		\(\cols(\Habs) \setminus \colsflx'''\)
		with first vertex \(w\) and last vertex \(w'\), with every colour used once.
		Therefore 
		\(Q = Q_1 \cup \Qabs \cup Q_2\) is a rainbow path, spanning in \(V(\Habs) \cup V'\) and \(\cols(\Habs) \cup \cols'\) with first vertex \(x\) and last vertex \(y\).
	\end{proof}

\section{Conclusion}
\label{sec:conclusion}
In this paper we considered the problem of containing a rainbow directed Hamilton cycle in uniformly coloured perturbed digraphs, when the number of colours is \(n\) (optimal) and the probability is \(\Theta(n^{-1})\) (optimal up to a constant factor).
A natural open problem to consider next is whether, under the same assumptions, it is possible to guarantee with high probability a rainbow copy of every fixed orientation of a Hamilton cycle, or even to prove a \emph{universality} result, that is showing that we can find a rainbow copy of every orientation simultaneously.
In the uncoloured setting such a universality result was recently obtained by Araujo, Balogh, Krueger, Piga and Treglown~\cite{simon-all-orientations}, for cycles of arbitrary lengths (not just Hamilton cycles). 
They proved that for every $\delta \in (0,1)$ there exists $C>0$, such that for every fixed orientation $D$ of a cycle of length between $3$ and $n$, with probability at least \(1-\e^{-n}\), the perturbed digraph \(D_0 \cup \dnp{n}{C/n}\)  (where as usual \(\delta n \) is a lower bound on the semidegree of \(D_0\)) contains $D$.
Taking the union bound over the at most \(n2^n\) choices for $D$ gives the universality result.
This greatly generalises the original result of Bohman, Frieze and Martin~\cite{bohman} where, under the same assumptions, they showed that with high probability the perturbed digraph has a directed Hamilton cycle.
It is plausible that our methods in combination with those of~\cite{simon-all-orientations} could extend the result of~\cite{simon-all-orientations} to the rainbow setting, i.e.\ that we can find with probability \(1-\e^{-n}\) a rainbow copy of a fixed oriented cycle of length \(\ell \in [2,n]\) when colouring uniformly with \(n\) colours.

	\bibliographystyle{amsplain} 

	\bibliography{main}   

\end{document}

%% file: preamble.tex
\usepackage{textcomp}
\usepackage{graphicx}

\usepackage{amssymb,amsmath,amsthm,color,mathtools}
\usepackage{float}
\usepackage[font=small]{caption}
\usepackage[font=footnotesize]{subcaption}
\usepackage{tikz}
\usepackage[pdftex,colorlinks,backref=page,citecolor=blue,linkcolor=blue,bookmarks=false]{hyperref}
\usepackage{comment}
\usepackage[shortlabels]{enumitem}
\usepackage{cleveref}
\usepackage{cite} 
\usepackage{pifont}

\usepackage{setspace}
\usepackage{tcolorbox}
\usepackage{xcolor}
\usepackage{framed}


%% file: macros.tex
\newcommand{\cols}{\calc}

\newtheorem{theorem}{Theorem}[section]
\newtheorem*{theorem*}{Theorem}
\newtheorem{lemma}[theorem]{Lemma}
\newtheorem*{lemma*}{Lemma}
\newtheorem{claim}[theorem]{Claim}
\newtheorem*{claim*}{Claim}
\newtheorem{proposition}[theorem]{Proposition}
\newtheorem*{proposition*}{Proposition}

\newtheorem*{corollary*}{Corollary}

\newtheorem{definition}[theorem]{Definition}

\newcommand{\prm}{^{\prime}}

\newcommand{\dprm}{^{\prime \prime}}

\DeclarePairedDelimiter\ceil{\lceil}{\rceil}

\newcommand{\abs}[1]{\left|#1\right|}

\newcommand{\nats}{\mathbb{N}}

\newcommand{\calc}{\mathcal{C}}
\newcommand{\expect}[1]{\mathbb{E}\left[#1\right]}

\newcommand{\tendsto}{\rightarrow}
\newcommand{\prob}[1]{\mathbb{P}\!\left[#1\right]}

\newcommand{\e}{{e}}

\newcommand{\Vflex}{V_{\text{flex}}}
\newcommand{\Cflex}{\cols_{\text{flex}}}

\newcommand{\calP}{\mathcal{P}}
\newcommand{\calH}{\mathcal{H}}
\newcommand{\calG}{\mathcal{G}}

\newcommand{\Vtr}{V_{\triangle}}
\newcommand{\Vsq}{V_{\square}}
\newcommand{\Vflx}{V_{\text{flex}\,}}

\newcommand{\Vbuf}{V_{\text{buf}\,}}

\newcommand{\colsflx}{\cols_{\text{flex}}}

\newcommand{\colsbuf}{\cols_{\text{buf}\,}}

\newcommand{\dnp}[2]{\mathbf{D}(#1,#2)}
\newcommand{\gnp}[2]{\mathbf{G}(#1,#2)}

\newcommand{\ddn}[2]{{D}_{#1,#2}}

\newcommand{\pathb}{\rho}
\newcommand{\gadgetb}{\nu}

\newcommand{\flxcnst}{\mu}
\newcommand{\coverub}{\zeta}

\newcommand{\eps}{\varepsilon}

\newcommand{\drgnp}[2]{\dir{\mathbf{G}}(#1,#2)}

\renewcommand{\drgnp}[2]{\mathbf{D}(#1,#2)}

\newcommand{\Habs}{H_{\text{abs}}}

\newcommand{\Np}{N^+}
\newcommand{\Nm}{N^-}
\newcommand{\dpp}{d^+}

\newenvironment{cframed}[1][blue]
  {\begin{tcolorbox}[colframe=#1,colback=white]}
  {\end{tcolorbox}}

\newcommand{\prblem}{\lambda}
\newcommand{\prbsqtr}{\prblem}
\newcommand{\prcon}{\prblem}

\newcommand{\Qabs}{Q_\text{abs}}

\newcommand{\mclb}{\gamma}
\newcommand{\goodmiddledges}{\mclb_3}

\newcommand{\remove}[1]{}